\documentclass[11pt]{article}

\usepackage{mathtools}
\usepackage{amsthm}
\usepackage{amsmath}
\usepackage{amssymb}
\usepackage{xcolor}
\usepackage{xypic}
\usepackage{cite}
\usepackage{float}
\usepackage{filecontents}
\usepackage{graphicx,caption}

\usepackage{epstopdf}
\usepackage{epsfig}
\usepackage{mathabx}

\theoremstyle{definition}
\newtheorem{definition}{Definition}[section]
\newtheorem{remark}[definition]{Remark}
\newtheorem{example}[definition]{Example}
\newtheorem{notation}[definition]{Notation}
\newtheorem{algorithm}[definition]{Algorithm}
\theoremstyle{plain}
\newtheorem{lemma}[definition]{Lemma}
\newtheorem{theorem}[definition] {Theorem}
\newtheorem{proposition}[definition] {Proposition}
\newtheorem{main}[definition]{Main Theorem}
\newtheorem{corollary}[definition]{Corollary}

\definecolor{forest}{rgb}{0.0, 0.5, 0.0}
\usepackage[colorlinks=true,linkcolor=blue,citecolor=forest]{hyperref}
\numberwithin{equation}{section}

\begin{document}
\title{Freiheitssatz for amalgamated products of free groups over maximal cyclic subgroups}
\author{Carsten Feldkamp}
\date{\today}
\maketitle
\abstract{In 1930, Wilhelm Magnus introduced the so-called Freiheitssatz: Let
$F$ be a free group with basis $\mathcal{X}$ and let $r$ be a cyclically reduced element of $F$ which
contains a basis element $x \in \mathcal{X}$, then every non-trivial element of the normal closure of $r$ in $F$ contains the basis element $x$. Equivalently, the subgroup freely generated by $\mathcal{X} \backslash \{x\}$ embeds canonically into the quotient group $F / \langle \! \langle r \rangle \! \rangle_{F}$. In this article, we want to introduce a Freiheitssatz for amalgamated products $G=A \ast_{U} B$ of free groups $A$ and $B$, where $U$ is a maximal cyclic subgroup in $A$ and $B$: If an element $r$ of $G$ is neither conjugate to an element of $A$ nor $B$, then the factors $A$, $B$ embed canonically into $G / \langle \! \langle r \rangle \! \rangle_{G}$.}

\section{Introduction}

For a group $G$ and an element $r \in G$ we denote the normal closure of $r$ in $G$ by $\langle \! \langle r \rangle \! \rangle_{G}$. We mostly write $G / \langle \! \langle r \rangle \! \rangle$ instead of $G / \langle \! \langle r \rangle \! \rangle_{G}$ if it is clear from the context, that the normal closure is taken over $G$. Further, we write $[a,b]=a^{-1}b^{-1}ab$ for the commutator $[a,b]$ of two elements $a,b$.

In 1930, W. Magnus proved the classical \emph{Freiheitssatz}: If $F$ is a free group with basis $\mathcal{X}$ and $r$ a cyclically reduced element containing a basis element $x \in \mathcal{X}$, then the subgroup freely generated by $\mathcal{X} \backslash \{x\}$ embeds canonically into the quotient group $F / \langle \! \langle r \rangle \! \rangle$. This result became a cornerstone of one-relator group theory and led to different kinds of natural generalizations. 

One way to generalize the Freiheitssatz of W. Magnus is to study so-called \emph{one-relator products}. A \emph{one-relator product} of groups $A_{j}$ ($j\in \mathcal{J}$) for some index set $\mathcal{J}$ is a quotient group $(\bigast_{j \in \mathcal{J}} A_{j} )/ \langle \! \langle r \rangle \! \rangle$, where $r$ is an element of $\bigast_{j \in \mathcal{J}} A_{j}$ which is not conjugate to an element of a single $A_{j}$ ($j \in \mathcal{J}$). Note that one-relator groups are special cases of one-relator products where the $A_{j}$ ($j\in \mathcal{J}$) are free groups. Thus, knowing the Freiheitssatz of W. Magnus, it is natural to ask under which conditions the factors $A_{j}$ ($j\in \mathcal{J}$) canonically embed into a one-relator product $(\bigast_{j \in \mathcal{J}} A_{j} )/ \langle \! \langle r \rangle \! \rangle$. One result in this context is the \emph{Freiheitssatz for locally indicable groups} (cf. Theorem~\ref{HowFS}) independently proved by J. Howie (see \cite{ArtHowOnpairs}), S. Brodskii (see \cite{ArtBrod80},\cite{ArtBrod84}) and H. Short (see \cite{PhDShort}). It is not known whether this generalized Freiheitssatz can be further generalized to torsion-free groups without stronger assumptions. However, several Freiheitss\"atze for one-relator products of torsion-free groups are known under restrictions to exponent sums of $r$ due to A. Klyachko (see \cite{kly93}) and under the condition that $r$ has a syllable length smaller or equal 
$8$ due to M. Edjvet and J. Howie (see \cite{edjhow20}). There are also further Freiheitss\"atze for one-relator products assuming small cancellation conditions on the symmetric closure of the relation $r$ (see \cite{ArtJuhaszFS}).

Another way of generalizing the Freiheitssatz of W. Magnus is to consider more than one relation $r$. Seeing the additional relations as part of the underlying group, these results are called \emph{Freiheitss\"atze for one-relator quotients}. Generalizations of that kind can be found in \cite{ArtHowSaeFS}, \cite{ArtHowSaeMagSub} due to J. Howie and M. Saeed. The aim of this article is to prove the following Freiheitssatz for special one-relator quotients which generalizes Chapter 4 of the author's dissertation \cite{DA}. After formulating this result we shortly discuss its assumptions and its connection to the results of \cite{ArtHowSaeFS}.

\begin{main} \label{main}
Let $A$ and $B$ be two free groups and let $G=A \ast_{U} B$ be an amalgamated product, where $U$ is a maximal cyclic subgroup in both factors. Further, let $r$ be an element of $G$ which is neither conjugate to an element of $A$ nor $B$. Then $A$ and $B$ canonically embed into the quotient group $G/ \langle \! \langle r \rangle \! \rangle_{G}$.
\end{main}

The following example shows that the assumption on $U$ to be a maximal cyclic subgroup cannot be omitted. 

\begin{example}
Let $A:= \langle a \mid \rangle$, $B:=\langle b,c,d \mid \rangle$ and $r:=a(bc)^{-1}$. Further, let $U$ be generated by $p:=a^2$ in $A$ and by $q := (bc)^{2}d^{2}$ in $B$ . Then we have in $(A \ast_{U} B)/\langle \! \langle r \rangle \! \rangle$:
\begin{eqnarray*}
(bc)^{2}d^{2} \ = \ a^{2} \ \ \Leftrightarrow \ \ (bc)^{2}d^{2} \ = \ (bc)^{2}  \ \ \Leftrightarrow \ \ d^{2} = 1
\end{eqnarray*}
Thus, $\langle c,d \mid \rangle$ does not embed into $(A \ast_{U} B)/\langle \! \langle r \rangle \! \rangle$ even though $q$ contains the basis element $b$ of $B$. 
\end{example}

Note that Main Theorem~\ref{main} is in part already contained in \cite[Theorems 3.1, Theorem 4.2 and Theorem 5.1]{ArtHowSaeFS} by J. Howie and M. Saeed since many amalgamated products of free groups over maximal cyclic subgroups in both factors are limit groups (in the sense of Z. Sela). For example the fundamental groups $\pi_{1}(S_{g}^{+})$ with $g \geqslant 2$ resp. $\pi_{1}(S_{g}^{+})$ with $g \geqslant 4$ of a compact, orientable resp. compact, non-orientable surface of genus $g$ are elementary equivalent to free groups and therefore limit groups. Other examples of amalgamated products of free groups over maximal cyclic subgroups in both factors that are limit groups can be found in the class of so-called \emph{doubles of free groups} $F \ast_{\langle w \rangle \cong \langle \widetilde{w} \rangle} \widetilde{F}$, where $w$ is an element of a free group $F$ and $\widetilde{F}, \widetilde{w}$ are copies of $F,w$. These doubles of free groups are word-hyperbolic if and only if $\langle w \rangle$ is a maximal cyclic subgroup in $F$ (see \cite{gorwil10}).

However, there are also many amalgamated products of free groups over maximal cyclic subgroups in both factors which are not limit groups. One such example is the group
\begin{eqnarray*}
G \ \ = \ \ \langle a,b,c,d,z \, \mid \, [a,b][c,d]=z^{4} \rangle \ \ = \ \ \langle a,b \mid \ \rangle \underset{\langle [a,b] \rangle \ \cong \ \langle z^{4}[c,d]^{-1} \rangle}{\ast} \langle c,d,z  \mid \ \rangle.
\end{eqnarray*}
In \cite{CCE91}, J. Comerford, L. Comerford and C. Edmunds have shown that a non-trivial product of two commutators in a free group can never be more than a cube. It follows that $z$ is in the kernel of every homomorphism $\varphi \colon G \rightarrow F_{n}$, where $F_{n}$ is the free group of rank $n$. Therefore, the finitely generated group $G$ cannot be $\omega$-residually free which is equivalent to $G$ not being a limit group (see \cite{ArtTheorie}).\\

In order to prove Main Theorem~\ref{main} we study special products of groups which we call tree-products (see Definition~\ref{deftreepr}) and finally introduce a Freiheitssatz for tree-products (see Theorem~\ref{thmFrTP}). Main Theorem~\ref{main} follows directly from this Freiheitssatz. As a tool we also prove the following lemma concerning one-relator products of locally indicable and free groups. Recall that a group $G$ is \emph{indicable} if there exists an epimorphism from $G$ to $\mathbb{Z}$. A group $G$ is \emph{locally indicable} if every non-trivial, finitely generated subgroup of $G$ is indicable.

\begin{lemma} \label{lemembl}
Let $H$ be a locally indicable group, $A$ be a free group with basis $\mathcal{A}$ and let $p:=v^{-1}gv$ be an element of $A$, where $v,g$ use letters from disjunct sets of basis elements from $\mathcal{A}$ and where $g_{j}$ is cyclically reduced as well as no proper power. Further, let $r$ be an element of $A \ast H$ which is not conjugate to an element of $A$ nor $H$ nor $\langle p \mid \rangle \ast H$. Then $\langle p \mid \rangle \ast H$ embeds canonically into $(A \ast H) / \langle \! \langle r \rangle \! \rangle$. 
\end{lemma}

We want to recall some results that are important for the proof of Main Theorem~\ref{main}. In 1981, J. Howie proved the following theorem which is known as the already mentioned \emph{Freiheitssatz for locally indicable groups}. This result was also proved independently by S. Brodskii (see \cite{ArtBrod80},\cite[Theorem~1]{ArtBrod84}) and by H. Short (see \cite{PhDShort}).

\begin{theorem}[\textbf{see \cite[Theorem 4.3 (Freiheitssatz)]{ArtHowOnpairs}}] \label{HowFS}
Suppose $G=(A \ast B) / N$, where $A$ and $B$ are locally indicable groups, and $N$ is the normal closure in $A \ast B$ of a cyclically reduced word $R$ of length at least $2$. Then the canonical maps $A \rightarrow G$, $B \rightarrow G$ are injective.
\end{theorem}

J. Howie obtains this Freiheitssatz by studying systems of equations over some group $G$. A finite system $\mathcal{W}$ of $m \in \mathbb{N}$ equations $w_{i}(x_{1},x_{2},\dots x_{n})$ ($1 \leqslant i \leqslant m$) in the variables $x_{1},x_{2},\dots x_{n}$ over the group $G$ corresponds to the presentation $S:=\langle G,x_{1},x_{2},\dots x_{n} \mid w_{1},w_{2},\dots,w_{m} \rangle$. The system $\mathcal{W}$ is said to \emph{have a solution over $G$} if it has a solution in some group containing $G$ as a subgroup. 

\begin{remark} [\textbf{cf. \cite[Proposition 2.3]{ArtHowOnpairs}}] \label{remmatrix}
The group $G$ embeds into $S$ if and only if the system $\mathcal{W}$ has a solution over $G$.
\end{remark}

Let $M$ be the $(m \times n)$ matrix whose $(i,j)$-th entry is the sum of the exponents of $x_{j}$ occurring in the word $w_{i} \in G \ast \langle x_{1} \rangle \ast \dots \ast \langle x_{n} \rangle$. A system $\mathcal{W}$ is called \emph{independent} if the associated matrix $M$ has rank $m$. It is conjectured that any independent system of equations over any group $G$ has a solution over $G$. J. Howie proved for example the following special case of that conjecture.

\begin{theorem} [\textbf{cf. \cite[Corollary 4.2]{ArtHowOnpairs}}] \label{thmHowequ}
Let $G$ be a locally indicable group. Then every independent system of equations over $G$ has a solution over $G$. 
\end{theorem}

Aside from Theorem~\ref{thmHowequ} we also use the following result about locally indicable groups.

\begin{theorem} [\textbf{see \cite[Theorem 4.2]{ArtHowOnLocInd}}] \label{Howlocind}
Let $A$ and $B$ be locally indicable groups, and let $G$ be the quotient of $A \ast B$ by the normal closure of a cyclically reduced word $r$ of length at least $2$. Then the following are equivalent:
\begin{itemize}
\item[$(i)$] $G$ is locally indicable;
\item[$(ii)$] $G$ is torsion-free;
\item[$(iii)$] $r$ is not a proper power in $A \ast B$.
\end{itemize}

\end{theorem}

\section{Tree-products}

First, we recall the definition of staggered presentations in the sense of W. Magnus from \textbf{\cite[Section~II.5]{BuchLS}} with some small alteration: Instead of using only cyclically reduced elements we allow some conjugation that leads to the definition of stabilizing generators. This is done in view of further applications and has nearly no influence on the presented proof. The alteration is not necessary for proving the main theorem of this article.

\begin{definition} [\textbf{staggered presentations over free groups, cf. \cite[Section~II.5]{BuchLS}}] \label{defstaggered}
Let \mbox{$G=\langle \mathcal{X} \cup \mathcal{S} \mid \mathcal{P} \rangle$} be a group presentation, $\mathcal{I} \subseteq \mathbb{Z}$ an index set and let $\mathcal{Y} = \bigsqcup_{i\in \mathcal{I}} \mathcal{Y}_{i}$ be a subset of $\mathcal{X}$. We assume $\mathcal{P}=\{v_{j}^{-1}g_{j}v_{j} \mid j \in \mathcal{J} \}$ for a totally ordered index set $\mathcal{J}$, $v_{j} \in \langle \mathcal{S} \mid \rangle$ and cyclically reduced elements $g_{j} \in \langle \mathcal{X} \mid \rangle$ such that each $g_{j}$ contains at least one element of $\mathcal{Y}$. For every $p_{j}:=v_{j}^{-1}g_{j}v_{j}$ we denote with $\alpha_{p_{j}}$ resp. $\omega_{p_{j}}$ the smallest resp. largest index $i$ such that $p_{j}$ contains an element of $\mathcal{Y}_{i}$. If $j<k$ implies $\alpha_{p_{j}} < \alpha_{p_{k}}$ as well as $\omega_{p_{j}} < \omega_{p_{k}}$, we say that $G= \langle \mathcal{X} \cup \mathcal{S} \mid \mathcal{P} \rangle$ is a \emph{staggered presentation} (\emph{over free groups}) and $\mathcal{P}$ a \emph{staggered set} of $\langle \mathcal{X} \cup \mathcal{S} \mid \, \rangle$. We call the generators from $\mathcal{S}$ \emph{stabilizing generators} and the generators from $\mathcal{X}$ \emph{non-stabilizing}.
\end{definition}

Next, we consider an example of a staggered set that is typical for the staggered sets arising in the proof of our main theorem.

\begin{example} \label{exstaggeredset}
We define
\begin{eqnarray*}
\mathcal{X} &=& \{a_{i} \mid i \in \mathbb{Z}\} \cup \{b_{i} \mid i \in \mathbb{Z}\} \cup \{c_{i} \mid i \in \mathbb{Z}\} \\
\mathcal{Y}_{0} &=& \{a_{0}\}, \ \ \mathcal{Y}_{1} \, = \, \{b_{1},c_{1}\}, \ \ \mathcal{Y}_{2} \, = \, \{a_{2}\}, \ \ \mathcal{Y}_{3} \, = \, \{a_{3},b_{3},c_{3}\}, \\ 
\mathcal{Y}_{4} &=& \{b_{4},c_{4}\}, \ \ \mathcal{Y}_{5} \, = \, \{a_{5}\} \ \ \text{and} \ \ \mathcal{Y}_{6} \, = \, \{a_{6}\}, \ \ \mathcal{S} \, = \, \emptyset.
\end{eqnarray*}
Further, let
\begin{eqnarray*}
p_{-2} \, = \, a_{0}b_{1}c_{1}^{-1}a_{0}a_{3}^{-1}, \ \ p_{0} \, = \, a_{2}b_{3}c_{3}^{-1}a_{2}a_{5}^{-1} \ \ \text{and} \ \ p_{1} \, = \, a_{3}b_{4}c_{4}^{-1}a_{3}a_{6}^{-1}.
\end{eqnarray*}
Then $P = \{p_{-2},p_{0},p_{1}\}$ is a staggered set of $\langle \mathcal{X} \mid \, \rangle = \langle \mathcal{X} \cup \mathcal{S} \mid \, \rangle$, where $\mathcal{Y} = \bigsqcup_{i=0}^{6} \mathcal{Y}_{i} \subset \mathcal{X}$.
\end{example}

\begin{definition} [\textbf{tree-products}] \label{deftreepr}
Let $\mathcal{T}$ be a tree. We associate a free group $F$ with basis $\mathcal{F} \neq \emptyset$ to every vertex of $\mathcal{T}$. Every edge of $\mathcal{T}$ between two vertices with vertex-groups $K$ and $L$ is associated to an \emph{edge-relation} $p=q$, where $p \in K$ and $q \in L$. We call $p$ resp. $q$ an \emph{edge-word} of $K$ resp. $L$ and say that $K$ and $L$ are connected over the edge-relation $p=q$. We demand that no edge-word is a proper power in its vertex-group and that for each vertex-group $F$ the set of all edge-words of $F$ forms a staggered set of $F$. Let $\mathfrak{G}$ be the union of the bases of all vertex-groups and let $\mathfrak{R}$ be the union of all edge-relations. Then we call the group $G:= \langle \mathfrak{G} \mid \mathfrak{R} \rangle$ a \emph{tree-product (associated to $\mathcal{T}$)}. We refer to a vertex-group that is associated to a leaf of $\mathcal{T}$ as a \emph{leaf-group}. If $\mathcal{T}$ only contains one vertex, we do not consider that vertex as a leaf.

We call a tree-product, associated to a subtree of $\mathcal{T}$, a \emph{subtree-product} of $G$. A subtree $\mathcal{B}$ of $\mathcal{T}$ is a \emph{branch} if we get a tree by deleting every vertex of $\mathcal{B}$ in $\mathcal{T}$ and every edge of $\mathcal{T}$ that is connected to a vertex of $\mathcal{B}$. We call a tree-product, associated to a branch of $\mathcal{T}$, a \emph{branch-product} of $G$.

Let $r$ be an element of $G$ and $A$ a leaf-group of $G$. We say that $r$ \emph{uses} the leaf-group $A$ with basis $\mathcal{A}$ if every presentation of $r$, written in the generators from $\mathfrak{G}$, uses at least one generator from $\mathcal{A}$.

Let the \emph{size} $|G|$ of $G$ be the number of its vertex-groups. Finally, let $\mathcal{P}$ be the set of the edge-words of all leaf-groups in $G$ and let $\mathcal{P}'$ be the set of all cyclical reductions of the elements of $\mathcal{P}$. Then we call $\sigma := \Sigma_{p \in \mathcal{P}'} |p|$ the \emph{boundary-length} of $G$.
\end{definition}

\begin{figure}[H]
\centering
\includegraphics[scale=0.28]{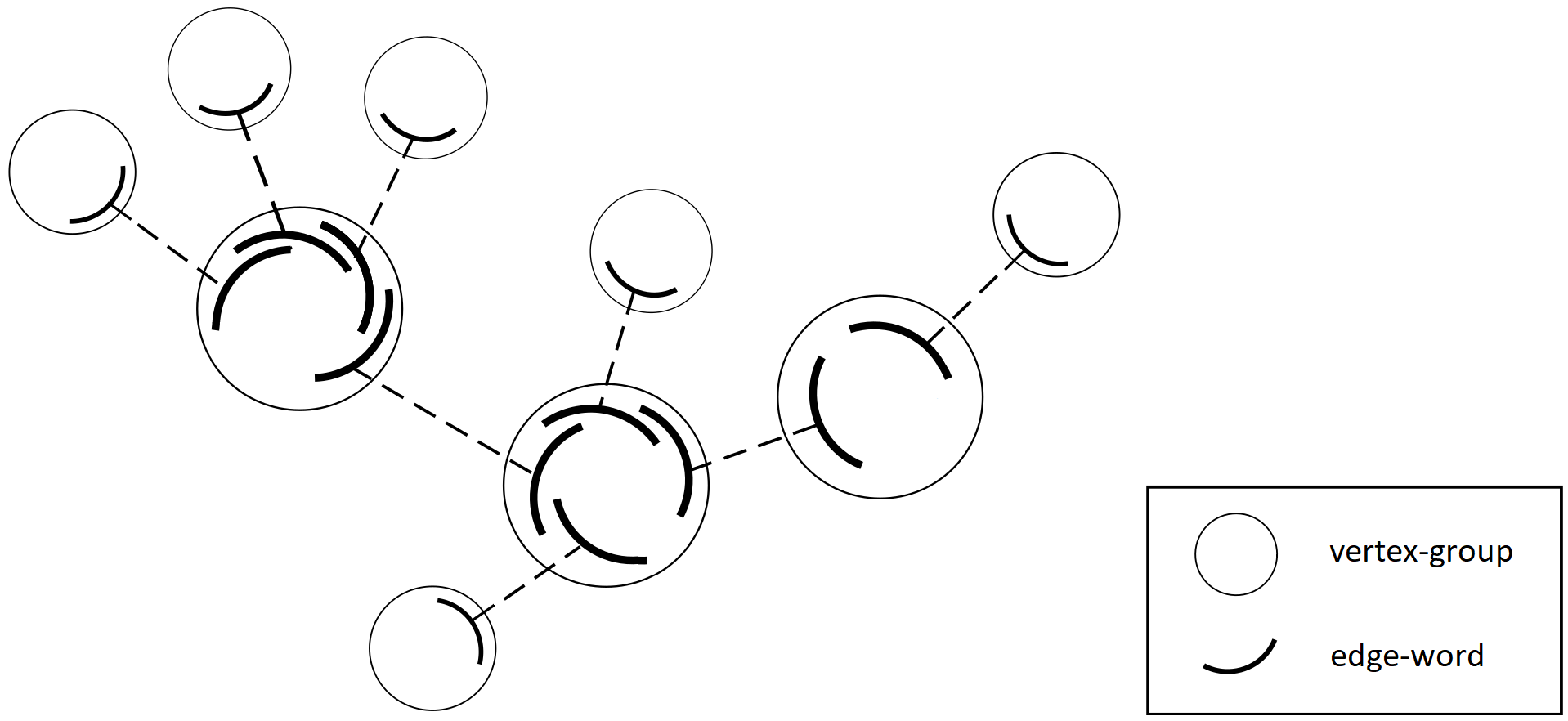}
\captionsetup{width=.8\linewidth}
\caption{Illustration for a tree-product of size 9 with 6 leaf-groups}
\label{picstaggeredset}
\end{figure}

\begin{notation}
Let $G$ be a tree-product and let $S$ be a subtree-product of $G$. We denote by $G \ominus S$ the free product of subtree-products of $G$ arising by deleting all generators and edge-relations of $S$ along with the edge-relations of all edges adjacent to $S$ from the presentation of $G$. Note that $G \ominus B$ is a tree-product for every branch-product $B$ of $G$. Analogously, we define $G \ominus \mathcal{M}$ for a free product $G$ of tree products and a set $\mathcal{M}$ of subtree-products of the tree-products from $G$. 
\end{notation}

\begin{lemma} \label{lemTPlocind}
Tree-products are locally indicable groups.
\end{lemma}

\begin{proof}
Let $G$ be a tree-product. Our proof is by induction over the size $|G|$. For the induction base $|G|=1$, we note that, according to Definition~\ref{deftreepr}, all vertex-groups of $G$ are free and therefore locally indicable. For the induction step let $G$ be a tree-product with size $n+1$ for some $n \in \mathbb{N}$. Further, let $A$ be a leaf-group of $G$ with edge-relation $p=q$, where $p \in A$. We have $G=(G \ominus A) \ast_{q=p} A$. By the induction hypothesis, $G \ominus A$ is locally indicable. Therefore, $G$ is locally indicable by Theorem~\ref{Howlocind}.
\end{proof}

\section{Contracted conjugates and minimal tree-products}

In this section, we find for an arbitrary element $r$ of a free product $G \ast H$, where $G$ is a tree-product and $H$ a locally indicable group, a uniquely determined minimal subtree-product $S$ of $G$ such that $S \ast H$ contains at least one conjugate of $r$. 

\begin{definition} \label{defmintreepr}
Let $G$ be a tree-product and $H$ a locally indicable group. Further let $r$ be an element of $G \ast H$ which is not conjugate to an element of $X \ast H$ for any vertex-group $X$ of $G$. We call a presentation of a conjugate $\widetilde{r}$ of $r$ \emph{contracted conjugate} if the following properties hold:

The presentation $\widetilde{r}$ contains under all presentations of all conjugates of $r$ only generators from $H$ and a minimal subtree-product $T$ of $G$. Further, let $\widetilde{r}$ be of the form $\Pi_{i=1}^{n} v^{(i)}$ fulfilling the following three properties.
\begin{itemize}
\item[(i)] Every $v^{(i)}$ contains either only basis elements of a single vertex-group or is an elements of $H \backslash \{1\}$. 
\item[(ii)] Cyclically proceeding $v^{(i)}$ ($i \in \mathbb{Z}_{n}$) contain elements of different vertex-groups or of a vertex-group and $H$.
\item[(iii)] No element $v^{(i)}$ is a power of an edge-word of a leaf-group of $T$.
\end{itemize}
We call the elements $v^{(i)}$ \emph{pieces} of the contracted conjugate $\widetilde{r}$. The number $n$ of pieces of $\widetilde{r}$ is the \emph{length} of $\widetilde{r}$ and is denoted by $||\widetilde{r}||_{G \ast H}$ (or $||\widetilde{r}||$). Finally, we call the minimal subtree-product $T$ \emph{minimal tree-product} of $r \in G \ast H$.
\end{definition}

The following lemma shows that Definition~\ref{defmintreepr} is well-defined.

\begin{lemma} \label{lemtreeprunique}
Let $G$ be a tree-product and $H$ be a locally indicable group. Further, let $r$ be an element of $G \ast H$ which is not conjugate to an element of $X \ast H$ for any vertex-group $X$ of $G$. Then the minimal tree-product of $r$ is uniquely determined.
\end{lemma}

\begin{proof}
Let $\widetilde{r}$ be a contracted conjugate of $r$ that is contained in a minimal tree-product $T$. For an arbitrary leaf-group $A$ of $T$ we consider the branch-products which arise from $G$ by deleting $T \ominus A$ along with all edge-relations of the edges that connect $T \ominus A$ to the rest of the tree-product. We denote the branch-product containing $A$ by $Z_{A}$ and the free product of all remaining branch-products by $R_{A}$. With the aim of obtaining a contradiction let $r^{*}$ be a contracted conjugate of $r$ that is contained in a minimal tree-product $T^{*}$ different from $T$. By definition, we have $|T|=|T^{*}|$. Since $T$ and $T^{*}$ are different there is at least one leaf-group $A$ of $T$ such that $T^{*}$ does not contain any vertex-group of $Z_{A}$. For such a leaf-group $A$ we define
\begin{eqnarray*}
X \ := \ (G \ominus R_{A}) \ast H, \ \ Y \ := \ (G \ominus Z_{A}) \ast H \ \ \text{and} \ \ S \ := \ (G \ominus Z_{A} \ominus R_{A}) \ast H.
\end{eqnarray*}
Note that 
\begin{eqnarray} \label{eqamalXY}
G \ast H \ = \ X \underset{S}{\ast} Y.
\end{eqnarray}
Since $X \cap (T^{*} \ast H)$ is a subgroup of $S$ and the tree-product $G \ominus Z_{A} \ominus R_{A}$ is smaller than $T$, the contracted conjugate $r^{*}$ is not an element of $X$. Thus, there exists an element $w \in G \ast H$ with
\begin{eqnarray*}
w^{-1} \widetilde{r} w \ = \ r^{*} \ \notin X.
\end{eqnarray*}
Since $\widetilde{r}$ is an element of $T \subset X$, $w$ is not an element of $X$. We consider the normal form $w=x_{1}y_{1}x_{2} \dots y_{m}x_{m+1}$ ($m \geqslant 1$) resp. the amalgamated product \eqref{eqamalXY}, where $x_{1}$, $x_{m+1} \in ( X \backslash S) \cup \{1\}$, $x_{i} \in X \backslash S$ for $2 \leqslant i \leqslant m$ and $y_{i} \in Y \backslash S$ for $1 \leqslant i \leqslant m$. We write
\begin{eqnarray*}
x_{m+1}^{-1} y_{m}^{-1} \cdots y_{1}^{-1} \underbrace{x_{1}^{-1} \widetilde{r} x_{1}}_{:= u} y_{1} x_{2} \cdots y_{m} x_{m+1} = r^{*} \in Y.
\end{eqnarray*}
For $x_{1}=1$ this equation contradicts the unique length of normal forms. For $x_{1} \neq 1$ we also get a contradiction because $u$ is a conjugate of $\widetilde{r}$ and can therefore not be an element of $S$.
\end{proof}

\section{Operations for tree-products}

In this section we introduce some operations for tree-products that will be helpful for the proof of our main theorem.

\begin{definition}[\textbf{root-products}] 
Let $G$ be a tree-product and $\mathcal{G}$ be the union of the bases of all vertex-groups of $G$. Further let $\mathcal{R}$ be the union of all edge-relations of $G$. We consider a subset $\mathcal{M} \subset \mathcal{G}$ that only contains basis elements of leaf-groups of $G$. For every element $a \in \mathcal{M}$ we choose an element $n(a) \in \mathbb{Z} \backslash \{0\}$. Then we say that
\begin{eqnarray*}
\widetilde{G} \ := \ \langle \mathcal{G} \cup \{ \widetilde{a} \mid a \in \mathcal{M} \} \ \mid \ \mathcal{R}, \ a = \widetilde{a}^{n(a)} \ (a \in \mathcal{M}) \rangle
\end{eqnarray*}
is a \emph{root-product} of $G$ which is obtained by extracting the \emph{roots} $\widetilde{a}$ of the elements $a \in \mathcal{M}$. 
\end{definition}

\begin{remark}
If the cyclical reductions of the edge-words of all the leaf-groups of $G$ contain at least two basis elements, every root-product of $G$ is also a tree-product.
\end{remark}

Before we introduce the next operation for tree-products we define special isomorphisms of tree-products:

\begin{definition}[\textbf{leaf-isomorphisms}] 
Let $G$ be a tree-product, $n \in \mathbb{Z} \backslash \{0\}$ and let $a,b$ be two stabilizing or two non-stabilizing generators of the same leaf-group $A$ of $G$. Then we refer to the transition from the old generator $b$ to the new generator $\bar{b} := ba^{n}$ or $\bar{b}:=a^{n}b$ as a \emph{leaf-isomorphism (of the tree-product $G$)} if the resulting group presentation is still a tree-product.
\end{definition}

\begin{remark} \label{remoneoptionok}
Note that in the stabilizing case both choices and in the non-stabilizing case at least one of the choices $\bar{b} := ba^{n}$ or $\bar{b}:=a^{n}b$ results in building a new tree-product since every edge-word  (written in the old generating set) is of the form $v^{-1}pv$ where $v$ only contains stabilizing generators and $p$ is cyclically reduced and no proper power.
\end{remark}

Root-products and leaf-isomorphisms can help to generate the conditions for applying the homomorphism described in the following definition (see Remark~\ref{remroot}).

\begin{definition}[\textbf{leaf-homomorphisms}] 
Let $G$ be a tree-product, $A$ be a leaf-group of $G$ and $H$ a locally indicable group. Let $p=q$ ($p \in A$) be the edge-relation of the edge connecting $A$ and $G \ominus A$. Further let $a$ be a basis element of $A$ which is contained in $p$ with exponential sum $0$. Then we call the homomorphism $\varphi \colon G \ast H \rightarrow \mathbb{Z}$ which maps $a$ to $1$  and all other generators of $G \ast H$ to $0$ a \emph{leaf-homomorphism} (for $a$). If we consider $G$ together with an element $r \in G \ast H$ that has exponential sum $0$ respectively $a$, we call $\varphi$ a \emph{leaf-homomorphism for $r \in G \ast H$}.
\end{definition} 

In the following, we are mostly interested in the kernels of leaf-homomorphisms.

\begin{remark} \label{remkernel}
The kernel $\ker(\varphi)$ of a leaf-homomorphism $\varphi \colon G \ast H \rightarrow \mathbb{Z}$ has the following structure: For $H$ we obtain countably infinitely many free factors $H_{i} := a^{-1} H a$ of $\ker(\varphi)$. We consider the free product of these factors as one locally indicable free factor $\widetilde{H}$ of $\ker(\varphi)$. For every generator $b \neq a$ of $A$ we obtain countably infinitely many generators $b_{i} := a^{-i} b a^{i}$ ($i \in \mathbb{Z}$). These generators build a  vertex-group $\widetilde{A}$ of $\ker(\varphi)$. For the edge-word $p \in A$ we obtain a staggered set $\mathcal{P} = \{p_{i} \mid i \in \mathbb{Z}\}$ of $\widetilde{A}$, where $p_{i} := a^{-i}pa^{i}$. Instead of $G \ominus A$ there are countable finitely many copies $X_{i} := a^{-i} (G \ominus A)a^{i}$ of $G \ominus A$. Every copy $X_{i}$ contains a copy $q_ {i} := a^{-i}qa_{i}$ of the edge-word $q$ of the edge-relation $p=q$ associated to the edge connecting $A$ with $G \ominus A$. We connect $X_{i}$ with $\widetilde{A}$ over the edge with edge-relation $p_{i}=q_{i}$ and thereby obtain a tree-product $K$ with $\ker(\varphi)=K \ast \widetilde{H}$. The edge-words $p_{i} \in \widetilde{A}$ ($i \in \mathbb{Z}$) are strictly shorter than the word $p \in A$ since all letters $a^{\pm 1}$ vanish and the word $p_{i}$ contains for every letter $b \neq a^{\pm 1}$ in $p$ exactly one letter $b_{j}$ ($j \in \mathbb{Z}$).  
\end{remark}

It is possible that a tree-product $G$ does not allow the application of a leaf-homomorphism. The following remark describes a procedure to either construct an isomorphic tree-product with strictly shorter boundary-length or to move on to a tree-product $\widetilde{G}$ allowing the application of a leaf-product such that $G$ can be embedded into $\widetilde{G}$. 

\begin{remark} \label{remroot}
Let $G$ be a tree-product, $A$ be a leaf-group of $G$ and $H$ a locally indicable group. Further, let $p=q$ ($p \in A$) be the edge-relation of the edge connecting $A$ with $G \ominus A$ and let $a$, $b$ be two non-stabilizing basis elements of $A$ which are contained in $p$ with exponential sums $p_{a}$, $p_{b} \neq 0$. First, we go over to the root-product $\widetilde{G}$ of $G$ setting $a = \widetilde{a}^{p_{b}}$. Afterwards, we consider the leaf-isomorphism $\psi$ of $A$ in $\widetilde{G}$ which is given by $\widetilde{b}= b\widetilde{a}^{p_{a}}$ or $\widetilde{b}= \widetilde{a}^{p_{a}}b$. It follows that $p$, written using the new generators $\widetilde{a}$ and $\widetilde{b}$ has exponential sum $0$ respectively $\widetilde{a}$. If $p$ does not contain the generator $\widetilde{a}$, the boundary-length of $\widetilde{G}$ respectively the new generating set is strictly smaller than the boundary-length of $G$. In the case that $p$ contains $\widetilde{a}$ there is a leaf-homomorphism $\varphi$ for the generator $\widetilde{a}$ of the leaf-group $A$ in $\widetilde{G}$. Since the only possible chances in the word lengths of $p \in G$ and $p \in \widetilde{G}$ come from the usage of the generators $a$/$\widetilde{a}$ and these generators vanish by building the kernel $\ker(\varphi)$ (see Remark~\ref{remkernel}), the length of $p_{i} \in \widetilde{A} \subset \ker(\varphi)$ is not only strictly shorter than the length of $p \in A \subset \psi(\widetilde{G})$, but also strictly shorter then the length of $p \in A \subset G$. 
\end{remark}

The following algorithm gives a contracted conjugate $\widetilde{r}^{*}$ in $\ker(\varphi)$ for a contracted conjugate $\widetilde{r}$ of an element $r$ in a tree-product $G$ that is contained in the kernel $\ker(\varphi)$ of a leaf-homomorphism $\varphi$ of $G$. In Lemma~\ref{lemalg}, we prove the functionality of that algorithm.

\begin{algorithm} \label{alg}
Let $G$ be a tree-product, $H$ a locally indicable group and $r$ an element of $G \ast H$ with contracted conjugate $\widetilde{r}$. Further let $r$ be an element of the kernel of a leaf-homomorphism $\varphi$ for a basis element $a$ of the leaf-group $A$ in $G$, where $\ker(\varphi)=K \ast \widetilde{H}$ is given with the structure described in Remark~\ref{remkernel}. The following algorithm rewrites the contracted conjugate $\widetilde{r}$ of $r \in G \ast H$ into a contracted conjugate $\widetilde{r}^{*}$ of $a^{-\ell}ra^{\ell} \in \ker(\varphi)$ ($ \ell \in \mathbb{Z}$) with $|| \widetilde{r}^{*}||_{\ker(\varphi)} \leqslant || \widetilde{r} ||_{G \ast H}$.

For reasons of symmetry, it suffices to only consider contracted conjugates $\widetilde{r}^{*}$ of $r=r_{0} \in \ker(\varphi)$: To find a contracted conjugate of $a^{-\ell}ra^{\ell} \in \ker(\varphi)$ ($\ell \in \mathbb{Z}$) we can apply the algorithm to $r=r_{0}$ and replace all indices $i$ ($i \in \mathbb{Z}$) in the resulting contracted conjugate $\widetilde{r}^{*}$ by $i+\ell$. 

Let $\widetilde{r}$ be given in the form $\widetilde{r} = \Pi_{i=1}^{n} v^{(i)}$ (see Definition~\ref{defmintreepr}). First, we step by step construct a presentation $\widetilde{r}'$ of $\widetilde{r}$  as an element of $\ker(\varphi)$. For the start we set $\lambda=0$, $\widetilde{r}'=1 \in \ker(\varphi)$ and $j = 1$.

\noindent\underline{Step 1.} We consider the $j$-th generator $e^{(j)}$ in $\widetilde{r}$. If $e^{(j)}=a$ we set $\lambda=\lambda-1$ and if $e^{(j)}=a^{-1}$ we set $\lambda=\lambda+1$. In the case $e^{(j)} \neq a^{\pm1}$ we add the letter $e_{\lambda}^{(j)}=a^{-\lambda} e^{(j)} a^{\lambda}$ to the word $\widetilde{r}'$. If $e^{(j)}$ is the last letter in $\widetilde{r}$, we go to Step 2. Elsewise we set $j=j+1$ and repeat Step 1 with the new input $\lambda$, $\widetilde{r}'$ and $j$.

The presentation $\widetilde{r}'$ given by Step 1 can be written in the form $\widetilde{r}'=\Pi_{i=1}^{n} \widetilde{v}^{(i)}$, where the subwords $\widetilde{v}^{(i)}$ for which $v^{(i)}$ does not lie in $A$ can be obtained from $v^{(i)}$ by adding the same index to every letter of $v^{(i)}$. Subwords $\widetilde{v}^{(i)}$ are trivial if and only if $v^{(i)}$ is a power $a^{k}$ ($k \in \mathbb{Z} \backslash \{0\}$). In that case the indices of the adjacent subwords $\widetilde{v}^{(i-1)}$ and $\widetilde{v}^{(i+1)}$ differ by $k$.

Now, we delete every trivial subword $\widetilde{v}^{(i)}$. If $\widetilde{v}^{(i)}$ is trivial and $\widetilde{v}^{(i-1)}$, $\widetilde{v}^{(i+1)}$ are elements of $\widetilde{H}$, we merge them into one subword. If $\widetilde{v}^{(1)}$ is trivial and $\widetilde{v}^{(2)}$, $\widetilde{v}^{(n)}$ are elements of $\widetilde{H}$, we permute the presentation $\widetilde{r}'$ cyclically by $\widetilde{v}^{(2)}$ and merge $\widetilde{v}^{(n)}\widetilde{v}^{(2)}$. We proceed analogously in the case that $\widetilde{v}^{(n)}=1$ and $\widetilde{v}^{(1)}$, $\widetilde{v}^{(n-1)} \in \widetilde{H}$. Let $\widetilde{r}'=\Pi_{i=1}^{n'} v'^{(i)}$ be the resulting presentation. Then $\widetilde{r}'$ satisfies the properties (i) and (ii) of Definition~\ref{defmintreepr}. Let $T'$ be the minimal subtree-product of $K$ such that $\widetilde{r}'$ only contains generators from $T'$ and $\widetilde{H}$. The following steps rewrite $\widetilde{r}'$ into a conjugate $\widetilde{r}^{*}$ of $r$ in $\ker(\varphi)$ that also satisfies property (iii) of Definition~\ref{defmintreepr}.

\noindent\underline{Step 2.} We replace all subwords $v'^{(i)}$ ($i \in \mathbb{Z}_{n'} = \{1,2,\dots,n'\}$) which are powers of edge-words $p$ of a leaf-group $L$ of $T'$ with edge-relation $p=q$ ($p \in L$) by the corresponding power of $q$. If $v'^{(i-1)}$ or $v'^{(i+1)}$ are elements of the vertex-group containing the edge-word $q$, we merge $v'^{(i)}$ with those subwords (after cyclically permutation if necessary). If there is a leaf-group of $T'$ such that no generator of $T'$ is used in the new presentation, we delete this leaf-group and the edge-relation of the edge connecting the leaf-group to the rest of $T'$. We repeat Step 2 with the new presentation and subtree-product $T'$ until there is no subword $v'^{(i)}$ left which is a power of an edge-word of a leaf-group of $T'$. Let $\widetilde{r}'' = \Pi_{i=1}^{n''} v''^{(i)}$ and $T''$ be the resulting presentation and subtree-product.

\noindent\underline{Step 3.} We step by step consider all leaf-groups $B$ of $T''$. Let $m$ be the number of all subwords $v''^{(i)}$ written in the basis of $B$ in the presentation $\widetilde{r}'' = \Pi_{i=1}^{n''} v''^{(i)}$. Further, let $\tau:\mathbb{Z}_{m}=\{1,2,\dots,m\} \rightarrow \mathbb{Z}_{n''}=\{1,2, \dots , n''\}$ be the function which maps $j$ to the position $i$ of the $j$-th subword written in the basis of $B$. We merge the subwords $v''^{(i)}$ from $(T'' \ominus B) \ast \widetilde{H}$ that are placed before the first, after the last or between two subwords $v''^{\tau(j)}$ ($j \in \{1,2,\dots,m\}$) to new subwords $w^{(j)}$ ($j \in \{0,1,\dots,m\})$. After conjugation with $w^{(0)}$ as well as merging $w^{(0)}$ and $w^{(m)}$ to a new $w^{(m)}$ we may write w.\,l.\,o.\,g.
\begin{eqnarray} \label{eqnf}
\widetilde{r}'' \ \ = \ \ \Pi_{j=1}^{m} v''^{(\tau(j))} w^{(j)}.
\end{eqnarray} 

In the case that there are subwords $w^{(i)}$ which correspond as elements of $\ker(\varphi)$ to powers of the edge-word $p \in B$, we merge them with their cyclical neighbours $v^{(\tau(i))}$ and $v^{(\tau(i+1))}$ to a new subword from $B$. After that, we repeatedly delete trivial subwords and merge newly cyclically adjacent $w^{(k)}$- or $v''^{(k)}$-subwords. Finally, we divide the remaining subwords from $(T'' \ominus B) \ast \widetilde{H}$ into the original subwords $v''(i)$ and return to Step 2 with the resulting presentation along with $T''$ as the new $T'$.

In the case that no subword $w^{(i)}$ of the presentation \eqref{eqnf} corresponds as an element of $\ker(\varphi)$ to a power of the edge-word $p \in B$, we repeat Step 3 with the next leaf-group of $T''$. When we checked all leaf-groups of $T''$ in one run of Step 3 without going back to Step 2, we set $\widetilde{r}^{*} := \widetilde{r}'' = \Pi_{i=1}^{n''} v''^{(i)}$.
\end{algorithm}

The following lemma secures the functionality of Algorithm~\ref{alg}.

\begin{lemma} \label{lemalg}
Algorithm~\ref{alg} ends after finitely many iterations. The element $\widetilde{r}^{*}$ given by Algorithm~\ref{alg} is a contracted conjugate of $a^{-i}ra^{i} \in \ker(\varphi)$ and the length of $\widetilde{r}^{*}$ is shorter or equal to the length of the contracted conjugate $\widetilde{r}$ of $r \in G \ast H$. 
\end{lemma}

\begin{proof}
The finiteness of the algorithm and the inequality $||\widetilde{r}^{*}||_{\ker(\varphi)} \leqslant || \widetilde{r} ||_{G \ast H}$ follow from the fact that no step increases the number of subwords $\widetilde{v}^{(i)}$/$v'^{(i)}$/$v''^{(i)}$ in the current presentation and every step of the algorithm which redirects to a previous step shortens the number of those subwords.

It remains to show that $\widetilde{r}^{*}$ is indeed a contracted conjugate of $r \in \ker(\varphi)$. First, we note that $r \in \ker(\varphi)=K \ast \widetilde{H}$ (see Remark~\ref{remkernel}) cannot be conjugate to an element of $X \ast \widetilde{H}$ for a vertex-group $X$ of $K$ since otherwise $r \in G \ast H$ would be conjugate to an element of $Y \ast H$, where $Y$ is a vertex-group of $G$ and therefore would not have a contracted conjugate (see Definition~\ref{defmintreepr}).

The validity of the properties (i) and (ii) for contracted conjugates (see Definition~\ref{defmintreepr}) is maintained by every step. Moreover, the validity of property (iii) is secured by Step 2. Note that the algorithm only ends if Step 3 leaves the presentation unchanged. So the output $\widetilde{r}^{*}$ satisfies the properties (i)-(iii) of Definition~\ref{defmintreepr}. We end the proof of this lemma by showing that no conjugate of $r$ in $\ker(\varphi)$ can be contained in the free product of $\widetilde{H}$ and a smaller subtree-product than $T''$. For this purpose we consider an arbitrary leaf-group $B$ of $T''$. Similar to the proof of Lemma~\ref{lemtreeprunique} we consider the branch-products arising from $K$ by deleting all vertex-groups of $T'' \ominus B$ along with the edge-relations of the adjacent edges. We denote the branch product containing $B$ by $Z_{B}$ and write
\begin{eqnarray} \label{eqamalpr1}
\ker(\varphi) \ \ = \ \ K \ast \widetilde{H} \ \ = \ \ Z_{B} \underset{p=q}{\ast} \big((K \ominus Z_{B}) \ast \widetilde{H} \big).
\end{eqnarray}
Note that in this last run of Step 3 the presentation~\eqref{eqnf} is for every leaf-group $B$ of $T''$ a cyclically reduced normal from of a conjugate of $r=r_{0} \in \ker(\varphi)$ respectively the amalgamated product
\begin{eqnarray*}
B \ \underset{p=q}{\ast} \ \big((T'' \ominus B) \ast \widetilde{H} \big)
\end{eqnarray*}
with length greater or equal two. This presentation is in particular a cyclically reduced normal form respectively the amalgamated product from \eqref{eqamalpr1} with length greater or equal two. Thus, no conjugate of $r$ in $\ker(\varphi)$ can be contained in $(K \ominus Z_{B}) \ast \widetilde{H}$. So the minimal tree-product of $r$ contains for every leaf-group $B$ of $T''$ at least one vertex-group of $Z_{B}$. Therefore the minimal tree-product of $r$ contains the tree-product $T''$. Since $\widetilde{r}^{*} \in T'' \ast \widetilde{H}$, the tree-product $T''$ is the minimal tree-product of $r$ in $\ker(\varphi)$.

\end{proof}

\begin{remark}
Let $G$ be a tree-product, $H$ a locally indicable group and $\widetilde{r}= \Pi_{i=1}^{n} v^{(i)}$ a contracted conjugate in $G \ast H$. Further, let $\varphi$ be a leaf-isomorphism from $G \ast H$ for a leaf-group $A$. Then, by writing every piece $v(i) \in A$ using the new generators from $\varphi(A)$, we get a contracted conjugate $\widetilde{r}^{*}$ of $\varphi(G \ast H)$ with the same length.
\end{remark}

The next lemma will allow us to apply the induction hypotheses in the proofs of our embedding theorems.

\begin{lemma} \label{lemlexi} Using the notation of Remark~\ref{remkernel} (and Remark~\ref{remroot}) let $r$ be a contracted conjugate in $G \ast H$, let $T$ be the minimal tree-product of an element $r \in \ker(\varphi)$ and let $\sigma$ be the boundary-length of $T$. We denote the contracted conjugate of $a^{-i}ra^{i} \in \ker(\varphi)$ by $r_{i}$. Further let $T_{i}$ be the minimal tree-product of $r_{i}$ and let $\sigma_{i}$ be the border length of $T_{i}$. Then we have
\begin{eqnarray*}
(||r_{i}|| - |T_{i}|, \sigma_{i}) \ \ < \ \ (||r|| - |T|,\sigma)
\end{eqnarray*}
respectively the lexicographical order (where the first component is weighted higher). 
\end{lemma}

\begin{proof}
Because of Lemma~\ref{lemalg}, we have $||r_{i}|| \leqslant ||r||$. Since every rooted tree is the union of all unique paths from the root to the leafs and $T_{i}$ contains for every leaf-group $L$ of $T$ at least one copy of $L$, we also have $|T_{i}| \geqslant |T|$, thus:
\begin{eqnarray*}
||r_{i}|| - |T_{i}| \geqslant ||r|| - |T| \ \Leftrightarrow \ ||r_{i}||-|T_{i}| = ||r|| - |T| \ \Leftrightarrow \ (||r_{i}||=||r|| \land |T_{i}|=|T|)
\end{eqnarray*}
Let $p=q$ (where $p \in A$) be the edge-relation to the edge connecting $A$ with $G \ominus A$. The equality $|T_{i}|=|T|$ implicates $T_{i}=(a^{-i}(T \ominus A) a^{i}) \ast_{q_{i}=p_{i}} \widetilde{A}$. As noticed in Remark~\ref{remkernel} and Remark~\ref{remroot} we have $|p_{i}|_{\widetilde{A}} < |p|_{A}$. Every other edge-word of $T_{i}$ is obtained from the associated edge-word of $T$ by adding indices. So the lengths of the edge-words different from $p$ do not change. Altogether, we get the inequality $\sigma_{i} < \sigma$ in the case $||r_{i}||-|T_{i}| = ||r|| - |T|$.
\end{proof}

We end this section with the following definition.

\begin{definition}[\textbf{reduction- and fan-generators}] \label{defredgen}
Let $\varphi \colon G \ast H \rightarrow \mathbb{Z}$ be a leaf-homomorphism for $r$ mapping a generator $a$ of a leaf-group $A$ of $G$ to $1$. We use the notation of Remark~\ref{remkernel} and consider the element $r$ as the element $r_{0}$ in $\ker(\varphi)$. If $r_{0}$ is an element of $(\widetilde{A} \underset{p_{i}=q_{i}}{\ast} X_{i}) \ast \widetilde{H}$ for an element $i \in \mathbb{Z}$, we call $a$ a \emph{reduction-generator} of $r \in G \ast H$. Elsewise we call $a$ a \emph{fan-generator} of $r \in G \ast H$.
\end{definition}

\section{Staggered presentations}

For the proof of the Magnus-Freiheitssatz and the Magnus property of free groups (see \cite{ArtMagnus}), W. Magnus defined \emph{staggered presentations} over free groups (cf. \cite[Section~II.5]{BuchLS}, also Definition~\ref{defstaggered}). We generalize that definition firstly to the case of locally indicable groups and secondly to the case of special tree-products. For the connection between the following definition and Definition~\ref{defstaggered} see Remark~\ref{remspecialfree}.

\begin{definition}[\textbf{staggered presentations over locally indicable groups}]  \label{defstagli}
Let $\mathcal{I}$, $\mathcal{J} \subset \mathbb{Z}$ be index sets, let $U$ be an arbitrary locally indicable group, let $V_{i}$ ($i \in \mathcal{I}$) be non-trivial locally indicable groups and let $r_{j}$ ($j \in \mathcal{J}$) be cyclically reduced elements of $U \ast \bigast_{i \in \mathcal{I}} V_{i}$, such that every $r_{j}$ uses at least one of the free factors $V_{k}$ ($k \in \mathcal{I}$). Further, let $W:=\ast_{i \in \mathcal{I}} V_{i}$. We denote the smallest (resp. largest) index $\ell \in \mathcal{I}$ such that $r_{j}$ uses the free factor $V_{\ell}$ by $\alpha_{r_{j}}$ (resp. $\omega_{r_{j}}$). We call $(U \ast W) / \langle \! \langle r_{j} \mid \mathcal{J} \rangle \! \rangle$ a \emph{staggered presentation (over locally indicable groups)} if the inequalities $\alpha_{r_{m}} < \alpha_{r_{n}}$ and $\omega_{r_{m}} < \omega_{r_{n}}$ hold for all $m$, $n \in \mathcal{J}$ with $m<n$.
\end{definition}

\begin{remark} \label{remspecialfree}
Definition~\ref{defstaggered} is apart from the stabilizing generators the special case of Definition~\ref{defstagli} for free groups. This becomes apparent in the following way: Let $G= \langle \mathcal{X} \cup \mathcal{S} \mid \mathcal{P} \rangle$ be a staggered presentation in the sense of Defintion~\ref{defstaggered} with respect to subsets $\mathcal{Y}_{i}$ ($i \in \mathcal{I}$) of $\mathcal{X}$. By cyclically reducing all elements of $\mathcal{P}$ we get a set $\mathcal{P}'=\langle r_{j} \mid j \in \mathcal{J} \rangle$ of cyclically reduced relations satisfying 
\begin{eqnarray*}
\langle \mathcal{X} \cup  \mathcal{S} \mid \mathcal{P} \rangle \ \ = \ \  \langle \mathcal{X} \cup  \mathcal{S} \mid \mathcal{P}' \rangle \ \ = \ \ \langle \mathcal{X} \mid \mathcal{P}' \rangle \ast \langle \mathcal{S} \mid \rangle.
\end{eqnarray*}
Defining $V_{i}:= \langle \mathcal{Y}_{i} \mid \rangle$ ($i \in \mathcal{I}$) and $U:= \langle (\mathcal{X} \cup  \mathcal{S}) \backslash \mathcal{Y} \mid \rangle$, where $\mathcal{Y} = \bigsqcup_{i\in \mathcal{I}} \mathcal{Y}_{i}$, we see that $\langle \mathcal{X} \cup  \mathcal{S} \mid \mathcal{P}' \rangle$ is a staggered presentation in the sense of Definition~\ref{defstagli} for the special case of free groups $V_{i}$ and $U$. 
\end{remark}

Before we continue examining staggered presentations we recall a definition of aspherical presentations.

\begin{definition}[\textbf{cf. \cite[Chapter III, Proposition 10.1]{BuchLS}}] \label{defasph}
A presentation $G= \langle  \mathcal{X} \mid  \mathcal{R} \rangle$ is \emph{aspherical} if and only if there are no non-trivial identities in $\langle \mathcal{X} \mid \rangle$ among the relations $ \mathcal{R}$.
\end{definition}

In the proofs of our embedding theorems we will use the following results.

\begin{theorem}[\textbf{cf. \cite[Chapter III, Proposition 11.1]{BuchLS}}] \label{thmasph1}
If $G = \langle  \mathcal{X} \mid  \mathcal{R} \rangle$ where $ \mathcal{R}$ consists of a single relator, or more generally, if the presentation is staggered (in the sense of Definition~\ref{defstagli}, but with free groups $V_{i}$ and $U$), then the presentation is aspherical.
\end{theorem}

\begin{remark} \label{rembasisinst}
Let $G= \langle \mathcal{X} \cup \mathcal{S} \mid \mathcal{P} \rangle$ be a staggered presentation in the sense of Defintion~\ref{defstaggered} with respect to subsets $\mathcal{Y}_{i}$ ($i \in \mathcal{I}$) of $\mathcal{X}$, where $\mathcal{P}= \{ p_{j} \mid j \in \mathcal{J}\}$ for an index set $\mathcal{J}$. Then the free group $\langle \mathcal{P} \mid \rangle$ embeds canonically into $\langle \mathcal{X} \cup \mathcal{S} \mid \rangle$, i.\,e. there are no non-trivial identities among the elements of $\mathcal{P}$ in $\langle \mathcal{X} \cup \mathcal{S} \mid \rangle$. This can easily be seen by the following observation: From Remark~\ref{remspecialfree} we notice that $G / \langle \! \langle \mathcal{S} \rangle \! \rangle = \langle \mathcal{X} \mid \mathcal{P}' \rangle$, where $\mathcal{P}'$ consists of the images of the elements $p_{j} \in \mathcal{P}$ under the canonical homomorphism from $G$ to $G / \langle \! \langle \mathcal{S} \rangle \! \rangle$ is a staggered presentation in the sense of Definition~\ref{defstagli}, but with free groups $V_{i}$ and $U$. Combining Definition~\ref{defasph} and Theorem~\ref{thmasph1} we see that there are no non-trivial identities among the elements of $\mathcal{P}'$ in $\langle \mathcal{X} \mid \rangle$. However, such an identity would follow from a non-trivial identity among the elements of $\mathcal{P}$ in $\langle \mathcal{X} \cup \mathcal{S} \mid \rangle$ by applying the homomorphism sending the elements of $\mathcal{X}$ to themselves and the elements of $\mathcal{S}$ to the trivial element.
\end{remark}

\begin{theorem}[\textbf{see \cite[Chapter III, Proposition 10.2]{BuchLS}}] \label{thmasph2}
If $G= \langle X \mid R \rangle$ is aspherical, and no element of $R$ is conjugate to another or to its inverse, then the following condition holds:
Let $p_{1} \cdots p_{n}=1$ where each $p_{i}=u_{i}r_{i}^{e_{i}}u_{i}^{-1}$ for some $u_{i} \in F$, $r_{i} \in R$, and $e_{i}=\pm 1$. Then the indices fall into pairs $(i,j)$ such that $r_{i}=r_{j}$, $e_{i}=-e_{j}$, and $u_{i} \in u_{j} N C_{i}$ where $C_{i}$ (the centralizer of $r_{i}$) is the cyclic group generated by the root $s_{i}$ of $r_{i}=s_{i}^{m_{i}}$. 
\end{theorem}

Next, we prove the following corollary of Theorem~\ref{HowFS} about staggered presentations of locally indicable groups.

\begin{corollary} \label{corstaggered}
Let $U$ be a locally indicable group and $(U \ast W) / \langle \! \langle r_{j} \mid \mathcal{J} \rangle \! \rangle$ be a staggered presentation for $W= \bigast_{i \in \mathcal{I}} V_{i}$ with locally indicable groups $V_{i}$. Further, let $w$ be a non-trivial element of $U \ast W$. Let $\alpha,\omega \in \mathcal{I}$ be indices such that $w$ is contained in the normal closure of the elements $r_{j}$ ($j \in \mathcal{J}$) in $U \ast W$ and uses only free factors $V_{k}$ with $\alpha \leqslant k \leqslant \omega$. Then $w$ is already contained in the normal closure of the elements $r_{\ell}$ in $U \ast W$ with $\alpha \leqslant \alpha_{r_{\ell}}$ and $\omega_{r_{\ell}} \leqslant \omega$. 
\end{corollary}

\begin{proof}
Let $V_{\leqslant \mu} := \bigast_{k \leqslant \mu} V_{k}$, $V_{\geqslant \mu} := \bigast_{k \geqslant \mu} V_{k}$ and $V_{\mu,\nu} := \bigast_{\mu \leqslant k \leqslant \nu} V_{k}$ for some indices $\mu,\nu \in \mathcal{I}$. We fix an arbitrary element $m \in \mathcal{J}$. First, we prove for an $n \in \mathcal{J}$ with $m \leqslant n$ the isomorphy
\begin{eqnarray} \label{eqiso1}
&& (U \ast W) / \langle \! \langle r_{k} \mid m \leqslant k \leqslant n \rangle \! \rangle \nonumber\\
&\cong & \big( (U \ast V_{\leqslant \omega_{r_{n-1}}}) / \langle \! \langle r_{k} \mid m \leqslant k \leqslant n-1 \rangle \! \rangle \big) \underset{U \ast V_{\alpha_{r_{n}},\omega_{r_{n-1}}}}{\ast} \big( ( U \ast V_{\geqslant \alpha_{r_{n}}}) / \langle \! \langle r_{n} \rangle \! \rangle \big)
\end{eqnarray}
by induction on $n-m$. For the induction base ($m=n$) the isomorphy follows directly from Theorem~\ref{HowFS}. Note for the induction step ($n \rightarrow n+1$) that $U \ast V_{\alpha_{r_{n+1}},\omega_{r_{n}}}$ embeds due to Theorem~\ref{HowFS} canonically into the factor $(U \ast V_{\geqslant \alpha_{r_{n}}}) / \langle \! \langle r_{n} \rangle \! \rangle$ which embeds due to the induction hypothesis into $(U \ast W) / \langle \! \langle r_{k} \mid m \leqslant k \leqslant n \rangle \! \rangle$. Since $U \ast V_{\alpha_{r_{n+1}}, \omega_{r_{n}}}$ also embeds due to Theorem~\ref{HowFS} canonically into $(U \ast V_{\geqslant \alpha_{r_{n+1}}}) / \langle \! \langle r_{n+1} \rangle \! \rangle$ we derive the desired isomorphy  
\begin{eqnarray*}
&& (U \ast W) / \langle \! \langle r_{k} \mid m \leqslant k \leqslant n+1 \rangle \! \rangle \\
&\cong & \big( ( U \ast V_{\leqslant \omega_{r_{n}}} ) / \langle \! \langle r_{k} \mid m \leqslant k \leqslant n \rangle \! \rangle \big) \underset{U \ast V_{\alpha_{r_{n+1}},\omega_{r_{n}}}}{\ast} \big( ( U \ast V_{\geqslant \alpha_{r_{n+1}}} ) / \langle \! \langle r_{n+1} \rangle \! \rangle \big).
\end{eqnarray*}
This ends the proof of the isomorphy \eqref{eqiso1}.

For the purpose of a contradiction, let $w$ be an element of $\langle \! \langle r_{j} \mid j \in \mathcal{J} \rangle \! \rangle_{U \ast W}$ which is also an element of the free product $H \ast V_{\alpha,\omega}$ for some indices $\alpha, \omega$, but is not contained in the normal closure $\langle \! \langle r_{\ell} \mid \alpha \leqslant \alpha_{r_{\ell}}, \omega_{r_{\ell}} \leqslant \omega \rangle \! \rangle_{U \ast W}$. Let us choose two indices $m,n \in \mathcal{J}$ such that $w$ is an element of the normal closure $\langle \! \langle r_{\ell} \mid m \leqslant \ell \leqslant n \rangle \! \rangle_{U \ast W}$ and $n-m$ is minimal with this property. We only consider the case $\omega < \omega_{r_{n}}$ since the case $\alpha_{r_{m}} < \alpha$ follows analogously. Consider the amalgamated product from isomorphy \eqref{eqiso1}. Because of the inequality $\omega < \omega_{r_{n}}$, $w$ is an element of the left factor of this amalgamated product. Since $w$ is trivial in $(U \ast W) / \langle \! \langle r_{k} \mid m \leqslant k \leqslant n \rangle \! \rangle$ it must already be trivial in $( U \ast V_{\leqslant \omega_{r_{n-1}}} ) / \langle \! \langle r_{k} \mid m \leqslant k \leqslant n-1 \rangle \! \rangle$. This contradicts the minimality of $n-m$.
\end{proof}

\section{Embedding theorems}

We begin this section by proving Lemma~\ref{lemembl}.\medskip

\noindent\textbf{Proof of Lemma~\ref{lemembl}.}
W.\,l.\,o.\,g. let $r$ be cyclically reduced. Our proof is by induction over the length $|g|$ of the cyclical reduction $g$ of $p$. For the base of induction let $|g|=1$. Then $p$ is a primitive element of $A$. We extend $p$ to a basis $\mathcal{A}'$ of $A$. Since $r$ is not an element of $\langle p \mid \rangle  \ast H$, $r$ contains a basis element from $\mathcal{A}' \backslash \{p\}$ which is a free basis of $A/ \langle \! \langle p \rangle \! \rangle$. We have
\begin{eqnarray*}
(A \ast H) / \langle \! \langle r \rangle \! \rangle \ \ = \ \ \big((A / \langle \! \langle p \rangle \! \rangle) \ \ast \ \langle p \mid \rangle \ \ast \ H \big) / \langle \! \langle r \rangle \! \rangle.
\end{eqnarray*}
So, by Theorem~\ref{HowFS}, $\langle p \mid \rangle \ast H$ embeds canonically into $(A \ast H) / \langle \! \langle r \rangle \! \rangle$.
For the induction case we choose two basis elements $a,b \in \mathcal{A}$ which are contained in $g$. We want to construct a basis-element $\widetilde{a}$ with $p_{\widetilde{a}}=0$. If $p_{a}=0$, we set $\widetilde{a}:=a,\widetilde{b}:=b$, and if $p_{a} \neq 0$, but $p_{b}=0$, we set $\widetilde{a}:=b,\widetilde{b}:=a$. For $p_{a} \neq 0 \neq p_{b}$ we define $\widetilde{a}$ by $a=\widetilde{a}^{p_{b}}$ and set $\widetilde{b}:=b\widetilde{a}^{p_{a}}$. Note that for $p$ written in the new basis $\mathcal{A}':=(\mathcal{A} \backslash \{a,b\}) \cup \{\widetilde{a},\widetilde{b}\}$ we have $p_{\widetilde{a}} =0$ in each case.

\noindent\underline{Case 1.} $r_{\widetilde{a}} = 0$\\
We consider the homomorphism $\varphi \colon A \ast H \rightarrow \mathbb{Z}$ sending $\widetilde{a}$ to 1 and every other basis element from $\mathcal{A}$ along with every element of $H$ to $0$. It is easy to see that the normal closure of $r \in A \ast H$ corresponds to the normal closure of the elements $r_{i}$ ($i \in \mathbb{Z}$) in $\ker(\varphi)$, where $r_{i}=\widetilde{a}^{-i} r \widetilde{a}^{i}$. We have
\begin{eqnarray*}
\ker(\varphi) \ \ = \ \ \widetilde{A} \ast \bigast_{i \in \mathbb{Z}} H_{i}, \ \text{ where } \ H_{i}:=\widetilde{a}^{-i} H \widetilde{a}^{i}
\end{eqnarray*}
and $\widetilde{A}$ is the free group with basis $\widetilde{\mathcal{A}} := \{x_{i} \mid x \in \mathcal{A}' \backslash \{\widetilde{a}\}, i \in \mathbb{Z}\}$ for $x_{i}=\widetilde{a}^{-i} x \widetilde{a}^{i}$. We define $p_{i}:=\widetilde{a}^{-i}p\widetilde{a}^{i}$ and $g_{i}:=\widetilde{a}^{-i}g\widetilde{a}^{i}$ ($i \in \mathbb{Z}$). To show the desired embedding it is sufficient to prove the embedding of $\langle p_{0} \mid \rangle \ast H_{0}$ into $\ker(\varphi) / \langle \! \langle r_{i} \mid i \in \mathbb{Z} \rangle \! \rangle$. By assumption, we have $r \in (A \ast H) \backslash H$. Thus, each $r_{i} \in \widetilde{A} \ast \bigast_{i \in \mathbb{Z}} H_{i}$ contains at least one piece of at least one free factor $H_{i}$. It follows that $\ker(\varphi) / \langle \! \langle r_{i} \mid i \in \mathbb{Z} \rangle \! \rangle$  is a staggered presentation over locally indicable groups (with $V_{i}:=H_{i}$, cf. Definition~\ref{defstagli}). Let $j \in \mathbb{Z}$ such that $r_{j}$ contains an element from $H_{0}$. By Corollary~\ref{corstaggered}, it remains to show the embedding of  $\langle p_{0} \mid \rangle \ast H_{0}$ into $\ker(\varphi) / \langle \! \langle r_{j} \rangle \! \rangle$. Because a free product of locally indicable groups is locally indicable, this embedding follows by the induction hypothesis which can be seen in the following way: Comparing $g$ written in the basis $\mathcal{A}$ and $g_{0}$ in the basis $\widetilde{\mathcal{A}}$ we see that every letter $a$ vanished without replacement and every other letter $x$ was replaced with a letter $x_{i}$. Since $g$ contained at least one letter $a$ by assumption, we have $|g|>|g_{0}|$. This justifies the application of the induction hypothesis.

\noindent\underline{Case 2.} $r_{\widetilde{a}}  \neq 0 \neq p_{\widetilde{b}}$\\
In this case the matrix
\begin{eqnarray*}
\begin{pmatrix}                                
p_{\widetilde{a}} & p_{\widetilde{b}} \\                                               
r_{\widetilde{a}} & r_{\widetilde{b}}                                               
\end{pmatrix}
\end{eqnarray*}
has full rank. Let $G:=A \ast_{p=c} (\langle c \mid \rangle \ast H)$. Because of Theorem~\ref{thmHowequ} and Remark~\ref{remmatrix} it follows that $\langle c \mid \rangle \ast H$ and therefore $\langle p \mid \rangle \ast H$ embeds canonically into $G / \langle \! \langle r \rangle \! \rangle = (\big (A \ast H) / \langle \! \langle r \rangle \! \rangle\big) \ast_{p=c} \langle c \mid \rangle$. Finally, we deduce that $\langle p \mid \rangle \ast H$ embeds canonically into $ (A \ast H) / \langle \! \langle r \rangle \! \rangle$.

\noindent\underline{Case 3.} $r_{\widetilde{a}}  \neq 0 = p_{\widetilde{b}}$\\
Note that because of $p_{\widetilde{b}}=0$ we were in the case $p_{a}=p_{b}=0$ when choosing $\widetilde{a}$ and defined $\widetilde{a}:=a,\widetilde{b}:=b$. Thus, if $r_{\widetilde{b}}=0$, we can apply Case 2 with reversed roles of $\widetilde{a}$ and $\widetilde{b}$. If $r_{\widetilde{b}} \neq 0$ we define $\bar{a}$ by $a=\bar{a}^{r_{b}}$ and set $\bar{b}:=b\bar{a}^{r_{a}}$. Note that for $r$, $g$ written using the basis elements of $\bar{\mathcal{A}}:=(\mathcal{A} \backslash \{a,b\}) \cup \{\bar{a},\bar{b}\}$ we have $r_{\bar{a}}=p_{\bar{a}} =0$. Thus, we can also apply Case 2; this time with $\bar{a}$ and $\bar{b}$ taking over the roles of $\widetilde{a}$ and $\widetilde{b}$. \qed\\

Next, we prove an easy lemma which enables us to pass from tree-products to certain root-products in the proofs of our embedding theorems.

\begin{lemma} \label{lemconsiderroot}
Let $H$ be a locally indicable group, let $G$ be a tree-product and let $\widetilde{G}$ be a root-product of $G$ which is obtained though the relation $a = \widetilde{a}^{k}$, where $k \in \mathbb{Z} \backslash \{0\}$ and $a$ is a generator of a leaf-group $A$. Further, let $r$ be an element of $G \ast H$ and let $U$ be a subgroup of $G$ such that $U \ast H$ embeds canonically into $( \widetilde{G} \ast H) / \langle \! \langle r \rangle \! \rangle$. Then $U \ast H$ embeds canonically into $(G \ast H) / \langle \! \langle r \rangle \! \rangle$. 
\end{lemma}

\begin{proof}
If $\langle a \rangle_{A}$ embeds into $( G \ast H) / \langle \! \langle r \rangle \! \rangle$, we set $m= \infty$. Elsewise let $m \in \mathbb{N}$ be the smallest power such that $a^{m}$ is trivial in $(G \ast H)/ \langle \! \langle r \rangle \! \rangle$. Then the desired statement follows directly from
\begin{eqnarray*}
( \widetilde{G} \ast H) / \langle \! \langle r \rangle \! \rangle \ \cong \ ( G \ast H ) / \langle \! \langle r \rangle \! \rangle \underset{a = \widetilde{a}^{k}}{\ast} \langle \widetilde{a} \mid \widetilde{a}^{km}=1 \rangle,
\end{eqnarray*}
since an element of $U \ast H$ which is trivial in the left factor of the amalgamated product must also be trivial in the amalgamated product itself.
\end{proof}

The following lemmata are further tools that will be used repeatedly in the proofs of our embedding theorems. For the proof of the first lemma we use a theorem of A. Karrass, W. Magnus and D. Solitar:

\begin{theorem}[\textbf{see \cite[Theorem 3]{ArtMaKaSo}}] \label{thmMaKaSo}
Let $G$ be a group with generators $a,b,c,\dots$ and a single defining relation $V^{k}(a,b,c,\dots)$, $k>1$, where $V(a,b,c,\dots)$ is not itself a true power. Then $V$ has order $k$ and the elements of finite order in $G$ are just the powers of $V$ and their conjugates. 
\end{theorem}

\begin{lemma} \label{lemconsiderT}
Let $G$ be a tree-product, $S$ be a subtree-product of $G$ and $H$ be a locally indicable group. Further, let $r$ be an element of $G \ast H$ with contracted conjugate $\widetilde{r}$ and minimal tree-product $T$. Finally, let $T$ contain at least one vertex-group of $S$.
If $\big(S \cap T \big) \ast H$ embeds canonically into $(T \ast H) / \langle \! \langle \widetilde{r} \rangle \! \rangle$, then $S \ast H$ embeds canonically into $(G \ast H) / \langle \! \langle r \rangle \! \rangle$.
\end{lemma}

\begin{proof}
First, we note that $T \cup S$ and $T \cap S$ are subtree-products of $G$ since $T$, $S$ are subtree-products of $G$ and $T$ contains by assumption at least one vertex-group of $S$. We may define
\begin{eqnarray} \label{eqp0}
P^{(0)} \ := \ \big( ( T \ast H) / \langle \! \langle \widetilde{r} \rangle \! \rangle \big) \underset{(S \cap T) \ast H}{\ast} \big( S \ast H \big),
\end{eqnarray}
where the embedding of the amalgamated subgroup into the left factor follows due to the assumption. Let $U^{(j)}$ ($j \in \{1,2,\dots,k\}, k \in \mathbb{N}$) be the branch-products obtained by deleting every vertex-group of $T$ and $S$ along with the edge-relations of the edges adjacent to $T$ or $S$. Further, let $p_{j}=q_{j}$ ($j \in \{1,2,\dots,k\}, k \in \mathbb{N}$) with $q_{j} \in U^{(j)}$ be the edge-relations of the edges connecting $U^{(j)}$ to $G \ominus U^{(j)}$. Let $\ell_{1} \in \mathbb{N} \cup \{\infty\}$ be the smallest power such that $p_{1}^{\ell_{1}}$ is trivial in $P^{(0)}$. We prove the following claim.
\begin{center}
\textit{For all $\ell \in \mathbb{N}$, $j \in \{1,2,\dots,k\}$ the element $q_{j}$ is of order $\ell$ in $U^{(j)} / \langle \! \langle q_{j}^{\ell} \rangle \! \rangle$.}
\end{center}
To prove that claim we consider the quotient group $\widetilde{U}^{(j)}$ of $U^{(j)}$ which we construct by taking the quotient with the normal closure of $q_{j}^{\ell}$ and every edge-word of $U^{(j)}$ apart from $q_{j}$. This quotient group is a free product of staggered presentations over free groups in the sense of Definition~\ref{defstaggered}. Note that by cyclically reducing all stabilizing generators we arrive at a staggered presentation of locally indicable groups in the sense of Definition~\ref{defstagli} for the special case of free groups (see Remark~\ref{remspecialfree}). Let $Q$ be the vertex-group of $U^{(j)}$ containing $q_{j}$. If $q_{j}^{\lambda}$ would be trivial in $U^{(j)}/\langle \! \langle q_{j}^{\ell} \rangle \! \rangle$ for some $\lambda \in \mathbb{N}$ with $\lambda < \ell$, then $q_{j}^{\lambda}$ would also be trivial in $\widetilde{U}^{(j)}$. Using Corollary~\ref{corstaggered} (for the special case of free groups) we conclude that $q_{j}^{\lambda}$ would be trivial in $Q / \langle \! \langle q_{j}^{\ell} \rangle \! \rangle$. This contradicts Theorem~\ref{thmMaKaSo}. Therefore, we can define inductively for $j \in \{1,2,\dots,k\}$
\begin{eqnarray*}
P^{(j)} \ \ = \ \ P^{(j-1)} \underset{p_{j}=q_{j}}{\ast} (U^{(j)} / \langle \! \langle q_{j}^{\ell_{j}} \rangle \! \rangle),
\end{eqnarray*}
where $\ell_{j} \in \mathbb{N} \cup \{\infty\}$ is the smallest power such that $p_{j}^{\ell_{j}}$ is trivial in $P^{(j-1)}$. 
Note that $P^{(k)} = (G \ast H) / \langle \! \langle r \rangle \! \rangle$. Altogether, we constructed $(G \ast H)/ \langle \! \langle r \rangle \! \rangle$ by iteratively building amalgamated products starting with the group $S \ast H$ (see \eqref{eqp0}). Thus, $S \ast H$ embeds canonically into $(G \ast H) / \langle \! \langle r \rangle \! \rangle$.
\end{proof}

\begin{lemma} \label{lemprim}
Let $G$ be a tree-product, $S$ a subtree-product of $G$ and $H$ a locally indicable group. Further, let $r$ be an element of $G \ast H$ with contracted conjugate $\widetilde{r}$. Under the condition that the minimal tree-product $T$ of $r$ contains at least one vertex-group of $S$ and assuming that at least one edge-word of a leaf-group of $T$ that is also part of $G \ominus S$ is a primitive element, $S \ast H$ embeds canonically into $(G \ast H) / \langle \! \langle r \rangle \! \rangle$. 
\end{lemma}

\begin{proof}
Due to Lemma~\ref{lemconsiderT} it is sufficient to consider the case that $G$ is the minimal tree-product $T$ of $r$. Let $A$ be a leaf-group of $G \ominus S$ and $G$ with edge-relation $p=q$, where $p \in A$ is a primitive element. Such a group exists by assumption. We extend $p$ to a basis of $A$ and replace the old basis with this new basis. Since the leaf-group $A$ is part of the minimal tree-product $T$ and all generators $p^{\pm 1}$ in $\widetilde{r}$ can be replaced with the help of the edge-relation $p=q$ we know that $\widetilde{r}$ contains at least one basis element $a \neq p$ of $A$. We have
\begin{eqnarray*}
(G \ast H) / \langle \! \langle r \rangle \! \rangle \ = \ \big( (G \ominus A) \underset{q=p}{\ast} \langle p \rangle \ast (A / \langle \! \langle p \rangle \! \rangle) \ast H \big) / \langle \! \langle r \rangle \! \rangle \ = \ \big( (G \ominus A) \ast H \ast (A / \langle \! \langle p \rangle \! \rangle) \big) / \langle \! \langle r \rangle \! \rangle.
\end{eqnarray*}
Because of Theorem~\ref{Howlocind} and $S \subseteq G \ominus A$, we get the desired embedding.
\end{proof}

The following proposition is the first step in proving our Freiheitssatz for tree-products (see Theorem~\ref{thmFrTP}).

\begin{proposition} \label{propemb}
Let $G$ be a tree-product and $S$ be a subtree-product of $G$. Further, let $H$ be a non-trivial locally indicable group and $r$ a contracted conjugate in $(G \ast H) \backslash G$. Assuming that the minimal tree-product of $r$ contains at least one vertex-group of respectively $S$ and $G \ominus S$, $S \ast H$ embeds canonically into $(G \ast H) / \langle \! \langle r \rangle \! \rangle$. 
\end{proposition}

\begin{proof}
Because of Lemma~\ref{lemconsiderT}, we may assume that $G$ is the minimal tree-product of $r$. Let $\sigma$ be the boundary-length of $G$ and let $|G \ominus S|$ be the number of all vertex-groups of $G \ominus S$. We prove the desired statement by induction over the tuple $(||r||-|G \ominus S|,\sigma)$ in lexicographical order (where the first component is weighted higher).

As the base of induction we consider the cases $||r||-|G \ominus S| \in \mathbb{Z}$, $\sigma=2$ and $||r||-|G \ominus S| \leqslant 0$, $\sigma \in \mathbb{N}$. If $|G \ominus S|=1$, let $A:=G \ominus S$. Elsewise we choose a leaf-group $A$ of $G$ and $G \ominus S$. In the case $\sigma=2$ the edge-word of $A$ is a primitive element of $A$. Therefore, we derive the desired embedding with Lemma~\ref{lemprim}. In the case $||r|| - |G \ominus S| \leqslant 0$ there is a vertex-group $B$ of $G \ominus S$ such that $r$ contains no basis element of $B$. The vertex-group $B$ cannot be a leaf-group of $G$ since we assumed that $G$ is the minimal tree-product of $r$. We consider the branch-products of $G$ which arise by deleting the vertex-group $B$ along with the edge-relations of the adjacent edges in $G$. Since $B$ is no leaf-group, there is at least one branch-product $Y$ among the resulting branch-products which is completely contained in $G \ominus S$. We denote the free product of the remaining branch-products by $R$. Let $p_{i}=q_{i}$ ($i \in \mathcal{I}$) with $p_{i} \in B$ be the edge-relations of the edges adjacent to $B$. Since, by assumption, $r$ is no element of $G$, Theorem~\ref{HowFS} gives us the canonical embedding of $Y \ast R$ into $(Y \ast R \ast H)/\langle \! \langle r \rangle \! \rangle$. Let $R^{(1)},R^{(2)}, \dots, R^{(k)}$ ($k \in \mathbb{N}$) be the free factors of $R$. Then every $q_{i}$ is contained in different factors of the free product $Y \ast R^{(1)} \ast R^{(2)} \ast \dots \ast R^{(k)}$. Because of Lemma~\ref{lemTPlocind} all free factors are locally indicable and therefore in particular torsion-free. Thus, the group $Q$ freely generated by $\{q_{i} \mid i \in \mathcal{I} \}$ embeds into $Y \ast R$ and hence into ($Y \ast R \ast H) / \langle \! \langle r \rangle \! \rangle$. Since the elements $p_{i}$ ($i \in \mathcal{I}$) of $B$ form a staggered set and $B / \langle \! \langle p_{i} \mid i \in \mathcal{I} \rangle \! \rangle$ is a staggered presentation over free groups, the group $P$ freely generated by $\{p_{i} \mid i \in \mathcal{I} \}$ embeds canonically into $B$ as noticed in Remark~\ref{rembasisinst}. So we may write:
\begin{eqnarray*}
(G \ast H) / \langle \! \langle r \rangle \! \rangle \ \cong \ B \underset{P \cong Q}{\ast} \big((Y \ast R \ast H) / \langle \! \langle r \rangle \! \rangle \big)
\end{eqnarray*}
Finally, note that $R \ast H$ embeds into the right factor of the amalgamated product by Theorem~\ref{HowFS}. Therefore, $S \ast H$ embeds into $(G \ast H) / \langle \! \langle r \rangle \! \rangle$.

For the induction step we consider the leaf-group $A$ of $G \ominus S$ and $G$ along with the edge-relation $p=q$ ($p \in A$) of the edge adjacent to $A$. If the cyclical reduction of $p$ contains only one basis element $a$ of $A$, we have $p=v^{-1}a^{\pm 1}v$, where $v$ consists only of stabilizing generators, due to Definition~\ref{deftreepr}. Thus, $p$ is a primitive element of $A$ and the desired embedding follows analogously to the induction base for $\sigma=2$. Let $p$ contain at least two different non-stabilizing basis elements $a$, $b$ of $A$. We recall the notation $p_{b}$ for the exponent sum of $p \in A$ respectively $b$.

If $p_{a}=p_{b}=0$, we consider the root-product $\widetilde{G}$ of $G$ given through $\widetilde{a}^{r_{b}}=a$, which is a tree-product because of $p \neq v^{-1}a^{\pm 1}v$. By Lemma~\ref{lemconsiderroot} it is sufficient to prove the embedding of $S \ast H$ into $( \widetilde{G} \ast H) / \langle \! \langle r \rangle \! \rangle$. We apply the leaf-isomorphism given by $\widetilde{b}:=b \widetilde{a}^{r_{a}}$ or $\widetilde{b} := \widetilde{a}^{r_{a}} b$ (cf. Remark~\ref{remoneoptionok}). In slight abuse of notation we denote the image of $\widetilde{G}$ under the leaf-isomorphism and the new contracted conjugate again by $\widetilde{G}$ and $r$. The new presentation of $r \in \widetilde{G} \ast H$ is also a contracted conjugate with the same length as the old presentation $r \in G \ast H$. Since $p_{\widetilde{a}}=0$ and $p$ is the only edge-word possibly containing $\widetilde{a}$, the exponent sum $r_{\widetilde{a}}$ is well-defined. We have $p_{\widetilde{a}}=r_{\widetilde{a}}=0$. If $p$ does not contain the generator $\widetilde{a}$, the tree-product $\widetilde{G}$ has a shorter boundary-length than $G$ and the desired embedding follows by the induction hypothesis. If $p$ contains $\widetilde{a}$, we go directly to Case 2. Thus, in the following part of the proof up to Case 2 we can assume that at least one exponential sum $p_{a}$ or $p_{b}$ is not $0$. W.\,l.\,o.\,g. let $p_{b} \neq 0$. Similar to the situation $p_{a}=p_{b}=0$ we construct through $a=\widetilde{a}^{p_{b}}$ the root-product $\widetilde{G} := G \ast_{a=\widetilde{a}^{p_{b}}} \langle \widetilde{a} \mid \rangle$ of $G$ and apply the leaf-isomorphism of $\widetilde{G}$ which is given by $\widetilde{b}:=b \widetilde{a}^{p_{a}}$ or $\widetilde{b} := \widetilde{a}^{p_{a}} b$ (cf. Remark~\ref{remoneoptionok}). We have $p_{\widetilde{a}}=0$. If $p$ does not contain the generator $\widetilde{a}$, the tree-product $\widetilde{G}$ has a shorter boundary-length than $G$ and the desired embedding follows by the induction hypothesis. Thus, assume that $p$ contains $\widetilde{a}$.\medskip

\noindent\underline{Case 1.} Let $r_{\widetilde{a}} \neq 0$.\\ 
By the preliminary considerations we have $p_{b} \neq 0$ and therefore also $p_{\widetilde{b}} \neq 0$. So (for some arbitrary fixed presentation $r$) the matrix
$\begin{pmatrix}                                
p_{\widetilde{a}} & p_{\widetilde{b}} \\                                               
r_{\widetilde{a}} & r_{\widetilde{b}}                                               
\end{pmatrix}$
has full rank and the desired embedding follows from Remark~\ref{remmatrix} and Theorem~\ref{thmHowequ}.

\noindent\underline{Case 2.} Let $r_{\widetilde{a}}=0$.\\
We consider the leaf-homomorphism $\varphi \colon \widetilde{G} \ast H \rightarrow \mathbb{Z}$ with $\varphi(\widetilde{a})=1$. Let $\mathcal{A}$ be the basis of the leaf-group $A$ in $G$. As noticed in Remark~\ref{remkernel}, we have $\ker(\varphi)=K \ast \widetilde{H}$, where $\widetilde{H} = \bigast_{\ell \in \mathbb{Z}} H_{\ell}$ for $H_{\ell} := a^{-\ell} H a^{\ell}$ ($\ell \in \mathbb{Z}$) and $K$ is a tree-product which can be constructed in the following way. We start with a vertex-group 
\begin{eqnarray*}
\widetilde{A} := \langle \widetilde{\mathcal{A}} \mid \rangle \ \ \text{for} \ \ \widetilde{A} := \{ \widetilde{a}^{-\ell} y \widetilde{a}^{\ell} \mid y \in \mathcal{A} \backslash \{a\}, \ell \in \mathbb{Z} \}
\end{eqnarray*}
and connect $\widetilde{A}$ over respectively one edge-relation $p_{\ell} = q_{\ell}$ ($p_{\ell} \in \mathcal{A}$) with countably infinitely many copies ($\widetilde{G} \ominus A)_{\ell} := a^{-\ell} ( \widetilde{G} \ominus A)a^{\ell}$ ($\ell \in \mathbb{Z}$) of $\widetilde{G} \ominus A$. Note that $|p_{\ell}|_{\widetilde{A}} < |p|_{A}$ for all $\ell \in \mathbb{Z}$ (cf. Remark~\ref{remroot}). 

By assumption, the contracted conjugate $r \in G \ast H$ uses the factor $H$. We define $r_{i} :=a^{-i} r a^{i}$. Then $\ker(\varphi) / \langle \! \langle r_{i} \mid i \in \mathbb{Z} \rangle \! \rangle$ is a staggered presentation for $V_{i} = H_{i}$ (cf. Definition~\ref{defstagli}). Algorithm~\ref{alg} gives us a contracted conjugates $r_{i}^{*}$ of $r_{i} \in \ker(\varphi)$ ($i \in \mathbb{Z}$) with $||r_{i}^{*}||_{\ker(\varphi)} \leqslant ||r||_{G \ast H}$ (see Lemma~\ref{lemalg}).  Since all elements of $S$ which are not contained in the kernel of $\varphi$ cannot be elements of the normal closure of $r$ in $\widetilde{G} \ast H$, it suffices to prove the embedding of the copy $S_{0} \ast H_{0} \subset (\widetilde{G} \ominus A)_{0} \ast H_{0}$ of $S \ast H$ in $\ker(\varphi) / \langle \! \langle r_{i}^{*} \mid i \in \mathbb{Z} \rangle \! \rangle$. By renaming, if necessary, let $r^{*}:=r_{0}^{*}$ be w.\,l.\,o.\,g. an element $r_{i}^{*}$ ($i \in \mathbb{Z}$) which uses the free factor $H_{0}$ of $\ker(\varphi)$. Then, by Corollary~\ref{corstaggered}, an element $w \in S_{0} \ast H_{0}$ is trivial in $\ker(\varphi)/\langle \! \langle r_{i}^{*} \mid i \in \mathbb{Z} \rangle \! \rangle$ if and only if it is trivial in $\ker(\varphi) / \langle \! \langle r^{*} \rangle \! \rangle$. The embedding that remains to show depends on the minimal tree-product $T$ of $r^{*}$:

\noindent\underline{Case 2.1.} $T$ contains at least one vertex-group of $S_{0}$.\\
Applying Lemma~\ref{lemconsiderT} on the subtree-product $S_{0}$ of $K$ we see that it suffices to consider $T \ast \widetilde{H}$ instead of $\ker(\varphi)$. So our aim is to prove the embedding of $(S_{0} \cap T) \ast H_{0}$ into $(T \ast \widetilde{H}) / \langle \! \langle r^{*} \rangle \! \rangle$. Note that $\widetilde{a}$ is either a reduction- or fan-generator (see Definition~\ref{defredgen}). If $\widetilde{a}$ is a reduction-generator of $r \in \widetilde{G}$, we have $T = \widetilde{A} \ast_{p_{0}=q_{0}} (\widetilde{G} \ominus A)_{0}$. As noticed above we have $|p_{0}|_{\widetilde{A}} < |p|_{A}$. Therefore, the boundary-length of $T$ is shorter than the boundary-length of $G$ and the desired embedding follows by the induction hypothesis.\\
Thus, assume that $\widetilde{a}$ is a fan-generator of $r \in \widetilde{G}$. We define $S' := S_{0} \cap T$. Using this notation our aim is to prove the embedding of $S' \ast H_{0}$ into $(T \ast \widetilde{H})/\langle \! \langle r^{*} \rangle \! \rangle$. We want to apply the induction hypothesis. Because of $||r^{*}||_{\ker(\varphi)} \leqslant ||r||_{G \ast H}$ it suffices to show $|T \ominus S'| > |G \ominus S|$. Note that for every leaf-group of $G$ there is at least one copy of this leaf-group in $T$ since $G$ is the minimal tree-product of $r$ and $T$ is the minimal tree-product of $r_{0}$. Because every rooted tree is the union of all unique paths from the root to the leafs, $T$ contains at least one copy of every vertex-group of $T \ominus S'$ and we get $|T \ominus S'| \geqslant |G \ominus S|$. By assumption, $\widetilde{a}$ is a fan-generator. Thus, $T$ contains at least two copies $C_{0}$, $C_{\mu}$ of the unique vertex-group $C$ of $G$ adjacent to $A$. If $C$ is a vertex-group of $G \ominus S$, the vertex-groups $C_{0}$ and $C_{\mu}$ are vertex-groups of $T \ominus S'$. If $C$ is a vertex-group of $S$ then $C_{\mu}$ is a vertex-group of $T \ominus S'$. In both cases we have an additional vertex-group of $T \ominus S'$ and therefore $|T \ominus S'|>|G \ominus S|$. By applying the induction hypothesis for ($||r^{*}||-|T \ominus S'|,\sigma'$), where $\sigma'$ is the boundary-length of $T$, we derive the desired embedding.

\noindent\underline{Case 2.2.} $T$ contains no vertex-group of $S_{0}$.\\
Let $S'$ be the free factor of $K \ominus T$ containing $S_{0}$. Further, let $D$ be the vertex-group of $K \ominus S'$ that is connected by an edge to $S'$ and let $d=e$ with $d \in D$ be the edge-relation of that edge. We first want to show that $D \ast \widetilde{H}$ embeds canonically into  $(T \ast \widetilde{H})/\langle \! \langle r^{*} \rangle \! \rangle$.\\
If $\widetilde{a}$ is a reduction-generator of $r \in \widetilde{G}$, we have $D=\widetilde{A}$, $T = \widetilde{A} \ast_{p_{j}=q_{j}} (\widetilde{G} \ominus A)_{j}$ with $j \in \mathbb{Z} \backslash \{0\}$ and $|p_{j}|_{\widetilde{A}} < |p|_{A}$, so the boundary-length of $T$ is shorter than the boundary length of $G$ (cf. Case 2.1). Because of $|T|=|G|$ and $|S| \geqslant 1$ we can therefore apply the induction hypothesis to conclude that $D \ast \widetilde{H}$ embeds canonically into $(T \ast \widetilde{H})/\langle \! \langle r^{*} \rangle \! \rangle$. In the case that $\widetilde{a}$ is a fan-generator we have $|T| \geqslant |G| +1$ so we can also apply the induction hypothesis to conclude the desired embedding of $D \ast \widetilde{H}$.\\
Using Lemma~\ref{lemconsiderT} it follows that $D \ast \widetilde{H}$ embeds into $((K \ominus S') \ast \widetilde{H})/\langle \! \langle r^{*} \rangle \! \rangle$. Thus, we may write
\begin{eqnarray*}
\ker(\varphi)/\langle \! \langle r^{*} \rangle \! \rangle \ \ = \ \ \big( ((K \ominus S') \ast \widetilde{H})/\langle \! \langle r^{*} \rangle \! \rangle \big) \ \underset{\mathbb{Z} \ast \widetilde{H}}{\ast} \ (S' \ast \widetilde{H}),
\end{eqnarray*}
where the generator of $\mathbb{Z}$ is mapped to $d$ in the left and $e$ in the right factor. Finally, we conclude that $S' \ast \widetilde{H}$ and therefore in particular $S_{0} \ast H_{0}$ embeds into $\ker(\varphi)/\langle \! \langle r^{*} \rangle \! \rangle$.
\end{proof}

\section{Freiheitssatz for tree-products}

Using Proposition~\ref{propemb} we prove the following Freiheitssatz, which in comparison to Proposition~\ref{propemb} omits the condition that $H$ is non-trivial.

\begin{theorem} [\textbf{Freiheitssatz for tree-products}] \label{thmFrTP}
Let $G$ be a tree-product, $S$ a subtree-product of $G$ and $H$ a (possibly trivial) locally indicable group. Further, let $r$ be an element of $G \ast H$ whose minimal tree-product contains at least one vertex-group of respectively $S$ and $G \ominus S$. Then $S \ast H$ embeds canonically into $(G \ast H) / \langle \! \langle r \rangle \! \rangle$. 
\end{theorem}

In a similar way to the proof of the Magnus-Freiheitssatz (cf. \cite{ArtMagnus}), the proof of Theorem~\ref{thmFrTP} will be simultaneous to the proof of Proposition~\ref{propmain}. Before we formulate this proposition, we introduce $\alpha$- (resp. $\omega$-)branch-limits and staggered presentations for tree-products analogously to Definition~\ref{defstagli}.

\begin{notation} \label{notKij}
Let $\mathcal{I} \subseteq \mathbb{Z} \cup \{\pm \infty\}$ be an index set, $\widetilde{H}$ be a locally indicable group, let $K$ be a tree-product and $Z_{i}$ ($i \in \mathcal{I}$) be branch-products of $K$ that share no common vertex-groups. For $m,n \in \mathbb{Z} \cup \{\pm \infty\}$ we define
\begin{eqnarray*}
K_{m,n} \ = \  K \ominus \{Z_{\ell} \mid \ell \in \mathcal{I} \land (\ell < m \ \lor \ n < \ell) \}.
\end{eqnarray*} 
\end{notation}

\begin{definition}[\textbf{staggered presentations over tree-products and {\boldmath $\alpha$}-/\boldmath{$\omega$}-branch-limits}]
Let $\mathcal{I}, \mathcal{J} \subseteq \mathbb{Z}$ be index sets, $\widetilde{H}$ be a locally indicable group, let $K$ be a tree-product and $Z_{i}$ ($i \in \mathcal{I}$) be branch products of $K$ that share no common vertex-groups. Further, let $r_{j}$ ($j \in \mathcal{J}$) be elements of $K \ast \widetilde{H}$ such that every minimal tree-product of these elements contains at least one vertex-group of respectively $Z_{i}$ and $K \ominus Z_{i}$ for at least one $i \in \mathcal{I}$. By $\alpha_{r_{j}}$ resp. $\omega_{r_{j}}$ we denote the greatest resp. smallest index $\lambda \in \mathbb{Z} \cup \{\pm \infty\}$ such that the minimal tree-product of $r_{j}$ is contained in $K_{\lambda,\infty} \ast \widetilde{H}$ resp. $K_{-\infty,\lambda} \ast \widetilde{H}$. We call $\alpha_{r_{j}}$ and $\omega_{r_{j}}$ the \emph{$\alpha$- (resp. $\omega$-)branch-limit} of $r_{j}$ in $K$. If we have $\alpha_{r_{m}} < \alpha_{r_{n}}$ and $\omega_{r_{m}} < \omega_{r_{n}}$ for all $m$, $n \in \mathcal{J}$ with $m<n$ we say that $(K \ast \widetilde{H}) / \langle \! \langle r_{j} \mid \mathcal{J} \rangle \! \rangle$ is a \emph{staggered presentation (over the tree-product $K$ with respect to the branch-products $Z_{i}$ ($i \in \mathcal{I}$))}.
\end{definition}

\begin{proposition} \label{propmain}
Let $(K \ast \widetilde{H}) / \langle \! \langle r_{j} \mid \mathcal{J} \rangle \! \rangle$ be a staggered presentation over a tree product $K$ with respect to branch-products $Z_{i}$ ($i \in \mathcal{I}$). Further, let $u$ be an element of the normal closure of the elements $r_{i}$ ($i \in \mathbb{Z}$) in $K \ast \widetilde{H}$, such that $u \in K_{\alpha,\omega} \ast \widetilde{H}$ for some $\alpha,\omega \in \mathbb{Z}$. Then $u$ is an element of the normal closure of the elements $r_{k}$ with $\alpha \leqslant \alpha_{r_{k}}$ and $\omega_{r_{k}} \leqslant \omega$ in $K \ast \widetilde{H}$.
\end{proposition}
 
\noindent\textbf{Simultaneous proof of Theorem~\ref{thmFrTP} and Proposition~\ref{propmain}.}\smallskip

\noindent We consider the tree-product $G$, the locally indicable group $H$ and the element $r$ from Theorem~\ref{thmFrTP}. Because of Lemma~\ref{lemconsiderT} it is sufficient to consider the case that $G$ is the minimal tree-product of $r$ and that $r$ is a contracted conjugate. Let $\sigma$ be the boundary-length of $G$. Our proof will be by induction over the tuple $(||r||-|G|,\sigma)$ in lexicographical order (where the first component is weighted higher). Let $T_{j}$ ($j \in \mathcal{J}$) be the minimal tree-products of the elements $r_{j}$ of Proposition~\ref{propmain}. W.\,l.\,o.\,g. we assume that the elements $r_{j}$ are contracted conjugates and that $\mathcal{J} = \mathbb{Z}$ or $\mathcal{J} =\{0,1,\dots,\delta\}$ for some $\delta \in \mathbb{N}_{0}$. Let $\sigma_{j}$ ($j \in \mathcal{J}$) be the boundary-lengths of the minimal tree-products $T_{j}$. First, we prove the following:\medskip

\noindent\underline{Statement 1.} For fixed $(t,s) \in \mathbb{Z} \times \mathbb{N}$ the statement of Theorem~\ref{thmFrTP} with $(||r||-|G|,\sigma) \leqslant (t,s)$ implicates the statement of Proposition~\ref{propmain} with $(\max_{j \in \mathcal{J}}(||r_{j}||-|T_{j}|),\sigma_{0}) \ \leqslant \ (t,s)$.\medskip

Let $u$ be a non-trivial element of the normal closure of the elements $r_{j}$ ($j \in \mathcal{J}$) in $K \ast \widetilde{H}$. Further, let $r_{j}$ ($m \leqslant j \leqslant n$) be the elements $r_{k}$ with $\alpha \leqslant \alpha_{r_{k}}$ and $\omega_{r_{k}} \leqslant \omega$. To get a contradiction we assume that $u$ is not contained in the normal closure of the elements $r_{i}$ ($m \leqslant i \leqslant n$) in $K \ast \widetilde{H}$. We choose indices $\zeta$, $\vartheta$ with $\zeta \leqslant \vartheta$ such that $u$ is contained in the normal closure of the elements $r_{k}$ ($\zeta \leqslant k \leqslant \vartheta$) in $K \ast \widetilde{H}$ and such that $\vartheta- \zeta$ is minimal with this property. Because of the assumption we have $\zeta < m$ or $n < \vartheta$. By inverting all indices, if necessary, we can assume w.\,l.\,o.\,g. $n < \vartheta$.

Our proof is by induction over $\lambda := \vartheta -\zeta \in \mathbb{N}_{0}$. For the base of induction ($\vartheta=0$), $u$ is an element of the normal closure of $r_{\vartheta}$ in $K \ast \widetilde{H}$. The inequality $n<\vartheta$ implicates the inequality $\omega < \omega_{r_{\vartheta}}$. Thus, $u$ is an element of $(K_{-\infty,\omega_{r_{\vartheta}}-1}) \ast \widetilde{H}$. Because of $(||r_{0}||-|T_{0}|,\sigma_{0})  \leqslant  (t,s)$, we derive with Theorem~\ref{thmFrTP} and Lemma~\ref{lemconsiderT} that $(K_{-\infty,\omega_{r_{\vartheta}}-1}) \ast \widetilde{H}$ embeds into $\big((K_{-\infty,\omega_{r_{\vartheta}}}) \ast \widetilde{H}\big) / \langle \! \langle r_{\vartheta} \rangle \! \rangle $. So $u$ has to be trivial in $K \ast \widetilde{H}$ which is a contradiction.

For the induction step ($\lambda \rightarrow \lambda+1$) we write
\begin{eqnarray*}
&& (K \ast \widetilde{H}) / \langle \! \langle r_{k} \mid \zeta \leqslant k \leqslant \vartheta \rangle \! \rangle\\
&\cong & \big((K_{-\infty,\omega_{r_{\vartheta}}-1} \ast \widetilde{H}) / \langle \! \langle r_{k} \mid \zeta \leqslant k \leqslant \vartheta-1 \rangle \! \rangle \big) \underset{K_{\alpha_{r_{\vartheta}},\omega_{r_{\vartheta}}-1} \ast \widetilde{H}}{\ast} \big( (K_{\alpha_{r_{\vartheta}},\infty} \ast \widetilde{H}) / \langle \! \langle r_{\vartheta} \rangle \! \rangle \big),
\end{eqnarray*}
where the embeddings of the amalgamated subgroup $K_{\alpha_{r_{\vartheta}},\omega_{r_{\vartheta}}-1} \ast \widetilde{H}$ follow from the induction hypothesis. Since $u$ is trivial in $K \ast \widetilde{H} / \langle \! \langle r_{k} \mid \zeta \leqslant k \leqslant \vartheta \rangle \! \rangle$ and is contained in the subgroup $K_{-\infty,\omega_{r_{\vartheta}}-1} \ast \widetilde{H}$, it is already trivial in the left factor $(K_{-\infty,\omega_{r_{\vartheta}}-1} \ast \widetilde{H}) / \langle \! \langle r_{k} \mid \zeta \leqslant k \leqslant \vartheta-1 \rangle \! \rangle$. This contradiction to the minimality of $\vartheta - \zeta$ ends the proof of Statement 1. 
 
Since the case $r \in (G \ast H) \backslash G$ is covered by Proposition~\ref{propemb}, we can further assume $r \in G$ for the proof of Theorem~\ref{thmFrTP}. Thus, we have $(G \ast H) / \langle \! \langle r \rangle \! \rangle = (G/\langle \! \langle r \rangle \! \rangle) \ast H$ and it suffices to prove the embedding of $S$ in $G / \langle \! \langle r \rangle \! \rangle$. Let $A$ be a leaf-group of $G$ which is not contained in $S$. Because of Statement 1 we cannot only use the induction hypothesis for Theorem~\ref{thmFrTP}, but also for Proposition~\ref{propmain}. Moreover, by proving Theorem~\ref{thmFrTP} we also prove Proposition~\ref{propmain}.

As the induction base we consider the cases $||r||-|G| < -1$ and $\sigma =2$. For $\sigma=2$ the edge-word of $A$ is a primitive element and the desired embedding follows from Lemma~\ref{lemprim}. In the case $||r||-|G| < -1$ there is a vertex-group $B$ of $G$ such that $r$ does not use a basis element of $B$ and $B$ is not directly connected to $A$. Note that $B$ is also no leaf-group of $G$ and can in particular not be $A$ since $G$ is the minimal tree-product of $r$. We consider the branch-products of $G$ which arise by deleting all generators of the vertex-group $B$ along with the edge-relations of the edges adjacent to $B$ in $G$. Let $Z_{A}$ be the branch-product which contains the leaf-group $A$ and let $R_{A}$ be the free product of the remaining branch-products. Since $B$ is not directly connected to $A$ we have $|Z_{A}| \geqslant 2$. Note that $Z_{A} \ast R_{A}$ is a tree-product using the fact that $R_{A}$ is locally indicable by Lemma~\ref{lemTPlocind}. We also note that $r$ uses at least the vertex-group $A$ of $Z_{A}$ and one vertex-group of $R_{A}$ because $G$ is the minimal tree-product of $r$. If $r \in Z_{A} \ast R_{A}$ possesses a minimal tree-product in $Z_{A}$, we get the canonical embedding of $(Z_{A} \ominus A)  \ast R_{A}$ into $(Z_{A} \ast R_{A})/ \langle \! \langle r \rangle \! \rangle$ by Proposition~\ref{propemb}. If $r \in Z_{A} \ast R_{A}$ does not possess a minimal tree-product in $Z_{A}$, we have $r \in A \ast R_{A}$ because of Definition~\ref{deftreepr}. Let $p=q$ with $p \in A$ be the edge-relation of the edge connecting $A$ and $Z_{A} \ominus A$. By Lemma~\ref{lemembl} we get the canonical embedding of $\langle p \mid \rangle \ast R_{A}$ into $(A \ast R_{A})/ \langle \! \langle r \rangle \! \rangle$. It follows
\begin{eqnarray*}
(Z_{A} \ast R_{A})/ \langle \! \langle r \rangle \! \rangle \ \ = \ \ \big((Z_{A} \ominus A)  \ast R_{A} \big) \ \underset{\mathbb{Z} \ast R_{A}}{\ast} \ \big((A \ast R_{A})/ \langle \! \langle r \rangle \! \rangle \big),
\end{eqnarray*}
where the generator of $\mathbb{Z}$ is mapped to $p$ in the right and to $q$ in the left factor. Altogether, we have the canonical embedding of $(Z_{A} \ominus A)  \ast R_{A}$ into $(Z_{A} \ast R_{A})/ \langle \! \langle r \rangle \! \rangle$ in every case. 

Let $u_{\ell}=v_{\ell}$ with $\ell \in \mathcal{L}$ for some index set $\mathcal{L}$ and $u_{\ell} \in B$ be the edge-relations of the edges adjacent to $B$ in $G$. Considering the free factors of $R_{A}$ as individual factors, every $v_{\ell}$ is contained in a different factor of $(Z_{A} \ominus A) \ast R_{A}$. Since all factors are locally indicable due to Lemma~\ref{lemTPlocind} they are in particular torsions-free. Thus, the free group $V$ with basis $\{v_{i} \mid i \in \mathcal{I} \}$ embeds into $(Z_{A} \ominus A) \ast R_{A}$. By Definition~\ref{defstaggered}, the elements $u_{i}$ ($i \in \mathcal{I}$) form a staggered-set of $B$ and $B/ \langle \! \langle u_{i} \mid i \in \mathcal{I} \rangle \! \rangle$ is a staggered presentation (over free groups). Because of Remark~\ref{rembasisinst} the free group $U$ with basis $\{u_{i} \mid i \in \mathcal{I} \}$ embeds canonically into $B$. Combining the embeddings proved so far we may write
\begin{eqnarray} \label{eq2amal}
G \ominus A  &\cong & B \underset{U \cong V}{\ast} \big( (Z_{A} \ominus A) \ast R_{A} \big) \nonumber \\
\text{and} \ \ \ G/ \langle \! \langle r \rangle \! \rangle & \cong & (G \ominus A) \underset{(Z_{A} \ominus A) \ast R_{A}}{\ast} \big( (Z_{A} \ast R_{A}) / \langle \! \langle r \rangle \! \rangle \big).
\end{eqnarray}
Now, the embedding of $G \ominus A$ and therefore in particular of $S$ into $G / \langle \! \langle r \rangle \! \rangle$ follows from the amalgamated product \eqref{eq2amal}.

For the induction step we consider the leaf-group $A$ and the edge-relation $p=q$ ($p \in A$) of the edge adjacent to $A$. The following preliminary considerations of the induction step as well as Case 1 are very similar to the corresponding part in the proof of Proposition~\ref{propemb}. In order to avoid unnecessary doubling we will shorten argumentations if they are already given in the proof of Proposition~\ref{propemb}. In the case that the cyclical reduction of $p$ contains only one basis element $a$ of $A$, $p$ is a primitive element of $A$ and the desired embedding follows analogously to the induction base for $\sigma = 2$. So we can assume that $p$ contains at least two different non-stabilizing basis elements $a$, $b$ of $A$.

If $p_{a}=p_{b}=0$, we consider the root-product $\widetilde{G}$ of $G$ given through $a=\widetilde{a}^{r_{b}}$ and apply the leaf-isomorphism given by $\widetilde{b}:=b \widetilde{a}^{r_{a}}$ or $\widetilde{b} := \widetilde{a}^{r_{a}} b$. Because of Lemma~\ref{lemconsiderroot} it is sufficient to prove the embedding of $S$ into $\widetilde{G} / \langle \! \langle r \rangle \! \rangle$. As usual, we denote in slight abuse of notation the image of $\widetilde{G}$ under the leaf-isomorphism and the new contracted conjugate again by $\widetilde{G}$ and $r$. In this situation we have $p_{\widetilde{a}}=r_{\widetilde{a}}=0$. If $p$ does not contain the generator $\widetilde{a}$, the tree-product $\widetilde{G}$ has shorter boundary-length than $G$ and the desired embedding follows by the induction hypothesis. If $p$ contains $\widetilde{a}$ we go directly to Case 2. Thus, in the following up to Case 2 we can assume that at least one exponential sum $p_{a}$ or $p_{b}$ is not $0$. W.\,l.\,o.\,g. let $p_{b} \neq 0$. We construct the root-product $\widetilde{G} := G \ast_{a=\widetilde{a}^{p_{b}}} \langle \widetilde{a} \rangle$ of $G$. Next, we apply the leaf-isomorphism of $\widetilde{G}$ which is given by $\widetilde{b} := b \widetilde{a}^{p_{a}}$ or $\widetilde{b} := \widetilde{a}^{p_{a}} b$ (cf. Remark~\ref{remoneoptionok}). It follows $p_{\widetilde{a}}=0$. With the same argumentation as before we may assume that $p$ contains $\widetilde{a}$. We consider two cases for $r_{\widetilde{a}}$.

\noindent\underline{Case 1.} Let $r_{\widetilde{a}} \neq 0$.\\
By the preliminary considerations we have $p_{\widetilde{b}} \neq 0$. Therefore, (for some arbitrary fixed presentation $r$) the matrix
$\begin{pmatrix}                                
p_{\widetilde{a}} & p_{\widetilde{b}} \\                                               
r_{\widetilde{a}} & r_{\widetilde{b}}                                               
\end{pmatrix}$
has full rank and the desired embedding follows from Remark~\ref{remmatrix} and Theorem~\ref{thmHowequ}.\medskip

\noindent\underline{Case 2.} Let $r_{\widetilde{a}} = 0$.\\
We consider the leaf-homomorphism $\varphi \colon \widetilde{G} \rightarrow \mathbb{Z}$ with $\varphi(\widetilde{a})=1$, define $K:=\ker(\varphi)$ and use the notations of Remark~\ref{remkernel} and Notation~\ref{notKij}. To show that $S$ embeds into $\widetilde{G} / \langle \! \langle r \rangle \! \rangle$ it suffices to show that the copy $S_{0}$ of $S$ embeds into $K / \langle \! \langle r_{i} \mid i \in \mathbb{Z} \rangle \! \rangle$. Because of Lemma~\ref{lemalg} we can assume w.\,l.\,o.\,g. that all $r_{i}$ ($i \in \mathbb{Z}$) are contracted conjugates. Let $Z$ be a leaf-group of $\widetilde{G}$ different from $A$ and let $Z_{i}:=\widetilde{a}^{-i} Z \widetilde{a}^{i}$. Since we assumed $G$ and therefore $\widetilde{G}$ to be the minimal tree-product of $r$ we know that each minimal tree-product $T_{i}$ of an element $r_{i} \in K$ ($i \in \mathbb{Z}$) contains at least one leaf-group $Z_{k}$ along with a vertex-group from $K \ominus Z_{k}$ for at least one $k \in \mathbb{Z}$. Thus, because of symmetry, $K / \langle \! \langle r_{i} \mid i \in \mathbb{Z} \rangle \! \rangle$ is a staggered presentation with respect to the branch-products $Z_{i}$ ($i \in \mathbb{Z}$). Let $j \in \mathbb{Z}$ be an index such that $T_{j}$ contains the leaf-group $Z_{0}$. Note that $S_{0}$ is part of $K_{0,0}$ (cf. Notation~\ref{notKij}). Thus, by Lemma~\ref{lemlexi}, we can apply the induction hypothesis for Proposition~\ref{propmain}. If $S_{0}$ does not contain the leaf-group $Z_{0}$, we immediately arrive at the desired embedding of $S_{0}$ into $K / \langle \! \langle r_{i} \mid i \in \mathbb{Z} \rangle \! \rangle$. Thus, we may assume in the following that $Z_{0}$ is part of $S_{0}$. By the induction hypothesis for Proposition~\ref{propmain} it is sufficient to show the canonical embedding of $S_{0}$ into $K / \langle \! \langle r_{j} \rangle \! \rangle$. Because of Lemma~\ref{lemconsiderT} it even suffices to prove the embedding of $S':=T_{j} \cap S_{0}$ into $T_{j}/\langle \! \langle r_{j} \rangle \! \rangle$.

Note that $\widetilde{a}$ has to be a reduction- or fan-generator (see Definition~\ref{defredgen}). If $\widetilde{a}$ is a reduction-generator, we have $T_{j}=\widetilde{A} \ast_{p_{0}=q_{0}} (\widetilde{G} \ominus A)_{0}$. As noticed in Remark~\ref{remroot} we also have $|p_{0}|_{\widetilde{A}} < |p|_{A}$. So the boundary-length of $T_{j}$ is strictly smaller than the boundary length of $\widetilde{G}$. This in combination with Lemma~\ref{lemalg} allows us to apply the induction hypothesis for Theorem~\ref{thmFrTP}. We get the desired embedding of $S'$ into $T_{j}/\langle \! \langle r_{j} \rangle \! \rangle$. It remains to consider the case that $\widetilde{a}$ is a fan-generator. In this case we have $|T_{j}| > |\widetilde{G}|(=|G|)$ since $T_{j}$ contains the vertex-group $\widetilde{A}$, at least one copy of every vertex-group of $\widetilde{G} \ominus A$ and at least two copies of the vertex-group $B$ of $\widetilde{G}$ adjacent to $A$. Combining this with Lemma~\ref{lemalg} we are able to apply the induction hypothesis for Theorem~\ref{thmFrTP} and deduce the embedding of $S'$ into $T_{j}/\langle \! \langle r_{j} \rangle \! \rangle$.\qed

\section{Proof of the Main Theorem}

In this section we prove Main Theorem~\ref{main} as a corollary of Theorem~\ref{thmFrTP} (Freiheitssatz for tree-products).\medskip

Let $p$ respectively $q$ be the generator of the subgroup of $A$ respectively $B$ which is identified with the amalgamated subgroup $U$. We choose bases $\mathcal{A}$ of $A$ and $\mathcal{B}$ of $B$ such that $p$ is cyclically reduced respectively $\mathcal{A}$ and $q$ is cyclically reduced respectively $\mathcal{B}$. Such bases are always available by replacing the basis elements with suitable conjugates if necessary. With the new bases, $G$ has the form of a tree-product of size $2$ (cf. Definition~\ref{deftreepr}). Since, by assumption, $r$ is neither conjugate to an element of $A$ nor $B$, $G$ is the minimal tree-product of $r$ (cf. Definition~\ref{defmintreepr}). Thus, the desired embeddings follow directly from Theorem~\ref{thmFrTP}.\\

\textbf{Acknowledgements.} I wish to thank Professor Jim Howie for his encouragement as well as his help to understand the connection between the results of \cite{ArtHowSaeFS} and the main theorem of this article.

\bibliography{Literatur}
    \bibliographystyle{alpha}

\end{document}